\documentclass[11pt,a4paper,english]{amsart}
\usepackage{babel}
\usepackage[latin1]{inputenc}
\usepackage{amsmath}
\usepackage{amsthm}
\usepackage{amsfonts}
\usepackage{indentfirst}
\usepackage{graphicx}

\usepackage{amssymb}
\usepackage[mathscr]{eucal}
\usepackage{tikz}

\newtheorem{de}{Definition}
\newtheorem{pro}{Proposition}
\newtheorem{cor}{Corollary}
\newtheorem{teo}{Theorem}
\newtheorem{rem}{Remark}
\newtheorem{lem}{Lemma}
\newtheorem{exa}{Example}

\renewcommand{\int}{{\rm int}}

\newcommand{\betabarra}{\bar{\beta}}

\newcommand{\bv}{{\bf v}}

\title{The log-canonical threshold of a  plane curve}
\author{C. Galindo, F. Hernando and F. Monserrat}
\curraddr{\texttt{Carlos Galindo and Fernando Hernando:} Instituto
Universitario de Matem\'aticas y Aplicaciones de Castell\'on and
Departamento de Matem\'aticas, Universitat Jaume I, Campus de Riu
Sec. 12071 Castell\'{o} (Spain).} \email{{\rm Galindo:}
galindo@mat.uji.es; {\rm Hernando:} carrillf@mat.uji.es}
\curraddr{\texttt{Francisco Monserrat:} Instituto Universitario de
Matem\'atica Pura y Aplicada, Universidad Polit\'ecnica de
Valencia, Camino de Vera s/n, 46022 Valencia (Spain).}
\email{framonde@mat.upv.es}
\date{}
\thanks{Supported by Spain Ministry of Economy
MTM2012-36917-C03-03 and University Jaume I P1-1B2012-04}

\begin{document}

\maketitle

\begin{abstract}
We give an explicit formula for the log-canonical threshold of a
reduced germ of plane curve. The formula depends only on the
first two maximal contact  values of the branches and their intersection
multiplicities. We also improve the two branches formula
given in \cite{Kuwata}.
\end{abstract}

\section*{Introduction}

Let $X_0$ be a nonsingular variety over an algebraically closed
field $k$ and consider a non-zero ideal sheaf $\mathfrak{a}
\subseteq \mathcal{O}_{X_0}$. Assume the existence of a log
resolution $\pi: X \rightarrow X_0$ of $\mathfrak{a}$ and let $F$ be the effective divisor such that
$\mathfrak{a} \mathcal{O}_{X} = \mathcal{O}_{X} (-F)$. For any rational number $t > 0$, the
multiplier ideal sheaf $\mathcal{J} (\mathfrak{a}^t)$ is defined to
be $\mathcal{J} (\mathfrak{a}^t) = \pi_* \mathcal{O}_X (K_{X|X_0} -
\lfloor t F \rfloor)$, where $K_{X|X_0}$ denotes the relative
canonical divisor of $\pi$ (that is, the unique exceptional divisor on $X$ such that ${\mathcal O}_X(K_{X|X_0})$ is the dual of the relative jacobian sheaf \cite[page 206 (2.3)]{l-s}) and $\lfloor \cdot \rfloor$ the
round-down or the integer part of the corresponding divisor. A
similar definition can be given for divisors on $X_0$. Multiplier ideals have
been introduced and studied for complex varieties and admit an
interesting analytic setting but the definition we use only depends on the
existence of a log resolution. Multiplier ideals have a precedent,
adjoint ideals which were introduced by Lipman in \cite{14adv}, and \cite[Chapters
9, 10, 11]{Lazarsfeld} is the main reference for them. They are an
important tool in birational geometry and singularity theory. Among other
reasons, it is worthwhile to mention that they provide information on the type of singularity
corresponding to the ideal $\mathfrak{a}$ and are very useful
because accomplish several vanishing theorems. However, explicit
computations are hard since they involve either to compute
resolutions of singularities or to obtain difficult integrals.

Attached to $\mathfrak{a}$, there exists an
increasing sequence of rational numbers $0=\iota_0 < \iota_1 <
\iota_2 < \cdots$, called jumping numbers of $\mathfrak{a}$, such
that $\mathcal{J} (\mathfrak{a}^j)= \mathcal{J}
(\mathfrak{a}^{\iota_l})$ for $\iota_l \leq j < \iota_{l+1}$ and
$\mathcal{J} (\mathfrak{a}^{\iota_{l+1}}) \subset \mathcal{J}
(\mathfrak{a}^{\iota_{l}})$ for each $l \geq 0$. That is, the
family of multiplier ideals of $\mathfrak{a}$ is totally ordered
by inclusion and parameterized by non-negative rational numbers
(see \cite{6-adv} for more information about the antecedents of
these numbers). Computing jumping numbers is not easy. They are
known for analytically irreducible (germs of) plane curves (see
\cite{jar,25-adv,16-adv} and there exists an algorithm for obtaining
them for ideals of the local ring at a rational singularity of a
complex algebraic surface \cite{24-adv}. A combinatorial criterium
for a rational number to be a jumping number of a complete ideal  $\mathfrak{a}$ of
finite co-length in a local regular 2-dimensional ring is given in \cite{HJ}.
The case when $\mathfrak{a}$ is simple is well-known since jumping numbers
and Poincar\'e series (an algebraic object relating jumping numbers and multiplier ideals)
have been computed \cite{jar, gal-mon}.

The first non-zero jumping number $\iota_1$ is named the {\it
log-canonical threshold} of $\mathfrak{a}$. This number is,
possibly, the most interesting of the jumping numbers; it appears
in many different contexts and is related with rather different
objects. Indeed, if $\mathfrak{a}$ is given by a polynomial
providing a complex hyper-surface germ with an isolated
singularity, the log-canonical threshold can be computed,
theoretically, via, the $L^2$ condition for holomorphic functions
\cite{Lazarsfeld}, the growth of the codimension of jet schemes
\cite{mus2-I}, the poles of the motivic zeta function \cite{66-B},
the generalized and twisted Bernstein-Sato polynomial
\cite{79-B,20-B}, the test ideals \cite{68-B}, the Arnold's
complex oscillation index \cite{79-B}, the Hodge spectrum
\cite{18-B} and, as we have said, the orders of vanishing on a log
resolution.

It seems that the first (implicit) use  of the
log-canonical threshold was in \cite{ati-I} where a conjecture of
Gelfand was proved. In the 1980's this number was seen as one of
the invariants considered by Steenbrink \cite{ste-I} in the
so-called spectrum of a singularity. It was named complex
singularity exponent and some of its properties were proved by
Varchenko \cite{var2-I,var3-I, var1-I}. Afterwards, Shokurov
\cite{sho-I} considered the log-canonical threshold within the
context of birational geometry. This concept allows us to define
log-canonical pairs which play a fundamental role in the Minimal
Model Program that recently has achieved a great advance
\cite{bchm-I}. The paper \cite{bir-I} on the Shokurov conjecture
(on the ACC for log-canonical thresholds of non-necessarily
nonsingular ambient spaces) has been crucial for that progress.

Despite of being a very interesting number, only a few explicit
generic computations of log-canonical thresholds are known. For
instance, as we have mentioned, it is known for analytically
irreducible (germs of) plane curves over algebraically closed
fields. It is $\frac{1}{\bar{\beta}_0} + \frac{1}{\bar{\beta}_1}$,
where $(\bar{\beta}_0,\bar{\beta}_1)$ are the first two maximal
contact values of the curve \cite{igusa, jar}. In the  complex
case, the log-canonical thresholds of irreducible quasi-ordinary
hyper-surface singularities are also known \cite{b-g-g} and the
motivic zeta function was used for this computation. However,
there is no formula for computing the log-canonical threshold of a
reduced plane curve. This is the goal of this paper. Uniquely, in
the complex case, a formula for the two-branches case is given in
\cite{Kuwata}. This formula depends on the  first two Puiseux
exponents (which are $\bar{\beta}_0$ and $\bar{\beta}_1)$ of the
branches and on its intersection number, and in many cases, the
log-canonical threshold is given as the minimum of three
candidates. Furthermore, in the same paper, it is also proved that
for any number of branches, the log-canonical threshold only
depends on the  first two Puiseux exponents of the branches and
their intersection numbers; however no formula is given. Finally,
in \cite{art,nai}, it has been proved that (even in the
non-reduced case), there exist suitable local coordinates  such
that the log-canonical threshold coincides with the  intersection
of the Newton polygon of the curve with the
diagonal line.

In this paper, we give in Theorem \ref{gordo2} an explicit formula for the log-canonical
threshold of a reduced (germ of) plane curve $C$. Our formula
extends that of \cite{Kuwata} to any number of branches. We also improve the formula of \cite{Kuwata} determining the value of the log-canonical threshold for the two branches case without having to calculate a minimum (see Corollary \ref{corol}). Since, in our case, there exists log resolution in any
characteristic and our methods do not depend on it, our formulae
hold for any reduced curve in any characteristic. As a byproduct of our result we provide a formula for the log-canonical threshold of a non-reduced plane curve and also for a (non-necessarily simple) complete ideal of finite co-length in the power series ring $k[[x,y]]$ (see Remarks \ref{rem5} and \ref{rem6}).

Next, we briefly describe the contents of the paper. Section
\ref{preliminaries} reviews some definitions and properties we need for stating
and proving our results; moreover, we give an example of a curve with 8
branches whose minimal log resolution is represented by its dual
graph, object that is an important tool in our development. Our
example is used in the paper to explain notations and results.
In Section \ref{results} we state our main result, Theorem
\ref{gordo2}, which is proved in Section \ref{proofs}. One step
in the proof is Proposition \ref{kuwatilla}, which asserts that to
compute the log-canonical threshold of $C$, we only need to take
into account those divisors of the minimal log resolution $\pi$ of
$C$ associated with the first two characteristic exponents \cite[III.2]{Campillo} of each branch.
For the complex case, this fact can be
deduced either from \cite{Kuwata} or \cite{FJ}. Our proof is valid for any
characteristic, it is essentially algebraic and supported in some results from Delgado in
\cite{Delgado}. Finally, we add that the formula in \cite{Kuwata} for the two branches case is a consequence of Theorem \ref{gordo2} and an improvement of that formula is stated as Corollary \ref{corol} in Section \ref{results}.

Theorem \ref{gordo2}, has two parts: The first one determines  an exceptional divisor $E_k$ of $\pi$ which provides the log-canonical threshold. The second part uses this index $k$ to show the exact value of the log-canonical threshold in terms of the intersection multiplicities between branches and the first two maximal contact values of each branch \cite[IV.2]{Campillo}.

Being more specific for part one, we define a weight function (see (\ref{sigma})) over those vertices $\mathbf{v}_j$ in the dual graph $\Gamma(C)$ of $\pi$ given by Proposition \ref{kuwatilla}. This weight is easily computable as a difference that involves  at most the first two maximal contact values of each branch. For getting the minuend and subtrahend, we delete $\mathbf{v}_j$ from $\Gamma(C)$ and distinguish among branches attached, or not, to the connected component containing the initial vertex $\mathbf{v}_1$.  Set $\mathcal{V}$ those vertices with weight and adapted degree (see the paragraph below Example \ref{examp4}) larger that two. Then,  the first part shows that $E_k$ corresponds to the end vertex of a distinguished path in $\Gamma(C)$. This path is the only one joining $\mathbf{v}_1$ with a vertex in $\mathcal{V}$ such that the weights of their vertices in $\mathcal{V}$ are negative, while the ones of the remaining vertices in $\mathcal{V}$ are not.

In Section \ref{proofs}, we give the proofs of Proposition
\ref{kuwatilla} and Theorem \ref{gordo2} by means of several
auxiliary results. Lemmas \ref{torero} to \ref{banderillero} allow us to show Proposition \ref{kuwatilla} and they are also useful for the proof of Theorem \ref{gordo2}. The first part of this theorem is proved with the help of three more lemmas which are uniquely based in combinatorics on the dual graph. The main step to show its second part is Lemma \ref{lemamadre}, where a comparison among candidates for log-canonical threshold is given. The proof of this lemma is supported on previous results of the paper and on a suitable choice of different partitions of the set of branches of the curve $C$.

\section{Preliminaries}
\label{preliminaries}

 Let $R=k[[x,y]]$ be the formal power series ring with coefficients over
 an algebraically closed field $k$. Consider a reduced series $f\in R$
and its decomposition $f=f_1f_2\cdots f_r$ as a product of
irreducible elements $f_i$ in $R$.  Denote by $C$ (respectively,
$C_i$, $1\leq i\leq r$) the divisors on ${\rm Spec}(R)$ (that is,
the germs of plane curves) defined by $f$ (respectively, $f_i$). Assume also that $C$ is singular and $r>2$ or $r=2$ but $C_1$ and $C_2$ are not non-singular and transversal. A \emph{log resolution} of $C$ is a composition of finitely many blowing-ups
centered at closed points
\begin{equation}\label{blow}
\pi:   X = X_{m} \stackrel{\pi_{m}}{\longrightarrow}
X_{m-1} \longrightarrow \cdots \longrightarrow X_{1}
\stackrel{\pi_{1}}{\longrightarrow} X_0 = {\rm Spec}(R),
\end{equation}
$E_j$ being the exceptional divisor of $\pi_j$, $1\leq j\leq m$, such that
$(f)\cdot {\mathcal O}_X={\mathcal O}_X(-F)$, where $F$ is an
effective divisor on $X$ which has simple normal crossing support.
By an abuse of notation, the strict transform of $E_j$ on $X$ will
also be denoted by $E_j$. From now on we will assume that $\pi$ is a
\emph{minimal} (with respect to the number of involved blowing-ups)
log resolution of $C$. Notice that it exists and is unique (see
\cite{Brieskorn}, for instance).

According to the Introduction, if $C$, $\pi$ and $f$ are as above,
we can associate with each positive rational number $t$ a {\it
multiplier ideal} $\pi_{*} {\mathcal O}_X(K_{X|X_0}-\lfloor t F
\rfloor)$ which we will denote by ${\mathcal J}(X_0,t C)$. The
\emph{log-canonical threshold} of $C$, ${\rm lct}(C)$, will be the
smallest positive rational number $\iota_1$ such that ${\mathcal
J}(X_0,\iota_1 C)\not=R$.

Let $O$ be the closed point of ${\rm Spec} (R)$ and ${\mathcal
C}=\{P_1=O,P_2,\ldots,P_m\}$ the set of closed points such that
$P_j$ is the center of $\pi_{j}$, $1 \leq j \leq m$; ${\mathcal
C}$ is a  \emph{constellation} of infinitely near points
\cite{2-nai} and ${\mathcal C}$ can be written as a union of
constellations ${\mathcal C}=\cup_{i=1}^r{\mathcal C}_i$ such
that, for each $i=1,\ldots,r$, ${\mathcal C}_i$ is the
constellation of points of $\mathcal C$ through which the successive strict transforms (by the blowing-ups in (\ref{blow})) of the branch
$C_i$ pass. Given two points $P_j, P_k\in {\mathcal C}$ with $k\geq j$, $P_k$ is called \emph{infinitely near} to $P_j$ (and
denoted $P_k\gtrsim P_j$) if the composition of blowing-ups
$X_{k-1}\rightarrow X_{j-1}$ maps $P_k$ to $P_j$; notice that $\gtrsim$ is a
partial ordering on $\mathcal C$. If, in addition, $P_k\neq P_j$ and $P_k$ belongs
to the strict transform of $E_j$ on $X_{k-1}$, then $P_k$ is said to
be \emph{proximate} to $P_j$, which is denoted by $P_k\rightarrow
P_j$. When $P_k$ is proximate to 2 points (respectively, 1
point) of $\mathcal C$ we name $P_k$ a \emph{satellite}
(respectively, \emph{free}) point.

The \emph{first infinitesimal neighborhood} of a point $P_j \in
\mathcal{C}$ is the family of closed points
belonging to the exceptional divisor obtained by the blowing-up at
$P_j$ and the \emph{$l$th infinitesimal neighborhood of $P_j$} ($l >
1$) is (inductively defined) the set of points in the first
infinitesimal neighborhood of some point in the $(l-1)$th
infinitesimal neighborhood of $P_j$.

It is well-known that $\pi^*C=\tilde{C}+\sum_{j=1}^m \mathbf{n}_j
E^*_j$, where $E_j^*$ is the total transform of $E_j$ and
$\tilde{C}$ the strict transform of $C$, both on $X$,
$\mathbf{n}_j$ being the multiplicity of the strict transform of $C$ at
$P_j$, $1\leq j\leq m$. We have also that
 $$E_j=E_j^*-\sum_{P_k\rightarrow P_j} E_k^*.$$
Writing $\pi^*C=\tilde{C}+\sum_{j=1}^m b_jE_j$ and $K_{X/X_0}=\sum_{j=1}^m a_j E_j$ (the relative
canonical divisor), the following equality holds (notice that $C$ is assumed to be singular):
\begin{equation}
\label{WW} {\rm lct}(C)=\min \left\{{\overline \alpha}_j :=
\frac{a_j+1}{b_j}\mid j=1,2,\ldots,m \right\}.
\end{equation} In some occasions, the values ${\overline \alpha}_j $ will be
named \emph{candidates} for log-canonical threshold of $C$. The
\emph{proximity matrix} of $\mathcal C$ is the matrix
$\mathbf{P}=(p_{kj})_{k,j=1}^m$ defined by $p_{kk}=1$ for all $k$,
$p_{kj}=-1$ if $P_k \rightarrow P_j$ and $0$ otherwise. Notice
that $(b_1, b_2, \ldots,b_m)^t:=\mathbf{P}^{-1}(\mathbf{n}_1,
\mathbf{n}_2, \ldots,\mathbf{n}_m)^t$ and $(a_1, a_2,
\ldots,a_m)^t:=\mathbf{P}^{-1}(1,1, \ldots,1)^t$.

The dual graph of $C$, $\Gamma(C)$, is an important object in our
development. It is an oriented tree such that the strict transform
(on $X$) of each exceptional divisor $E_j$ is represented by a
vertex, ${\bf v}_j$.
Two vertices are connected by an edge if the
corresponding divisors meet. The strict transform (on $X$) $\tilde{C}_i$ of each component $C_i$, $i=1,\ldots,r$, is represented by an arrow, ${\bf a}_i$, which is a label of the vertex associated with the exceptional divisor whose strict transform meets $\tilde{C}_i$ (or, in other words, the maximum point of $\mathcal{C}_i$ for the
ordering $\lesssim$). Usually each vertex ${\bf v}_j$ is
labeled by the number $j$ (that is, the number of blowing-ups in the sequence (\ref{blow}) needed
to create the corresponding divisor). The initial vertex of the edge
that joins two vertices ${\bf v}_{j_1}$ and ${\bf v}_{j_2}$ is the
one labeled with $\min\{j_1,j_2\}$.

We use the dual graph complemented with the so-called proximity
(oriented) graph of $C$ for representing the minimal log resolution of $\pi$.
The proximity graph allows us to decide which vertices are involved
in the resolution of each branch. Its vertices correspond with the
points in $\mathcal{C}$ and its edges join proximate points. An edge
joining $P_k$ and $P_j$ $(k >j)$ is a continuous straight line
whenever $P_k$ is in the first infinitesimal neighborhood of $P_j$,
otherwise it is a dotted curved line. By convention, we will omit those dotted curved edges which are deduced from others. As in the dual graph, we label
with an arrow ${\bf a}_i$ the vertices corresponding with the maximum
point in $\mathcal{C}_i$.


\begin{exa}\label{example1}
{\rm

Figure 1 shows the proximity and dual graphs (where we have omitted the orientation)
of a reduced curve $C$
with 8 components, $C_1, C_2,\ldots,C_8$, defined by 8 irreducible
elements $f_1, f_2, \ldots,f_8\in R$, such that its minimal log resolution
$\pi$ is obtained by blowing-up a constellation, ${\mathcal
C}=\{P_j\}_{j=1}^{17}$, of 17 infinitely near points. The reader can see additional
labels in some vertices of the dual graph that will be explained in
forthcoming examples. One has that ${\mathcal C}=\cup_{i=1}^8
{\mathcal C}_i$, where
$${\mathcal C}_1=\{P_1,P_2,P_3,P_4,P_5,P_6,P_7\},\;\;{\mathcal C}_2=\{P_1,P_2,P_{3},P_{4},P_{5},P_{8}\},$$
$${\mathcal C}_3=\{P_1,P_2,P_{3},P_{4}, P_9, P_{10}, P_{11}\},\;\; {\mathcal C}_4=\{P_1,P_2,P_{3},P_{4}, P_9, P_{10}, P_{12},P_{13} \},$$
$${\mathcal C}_5=\{P_1,P_2,P_{16},P_{17}\}\;\;\mbox{ , }\;\; {\mathcal C}_6={\mathcal C}_7=\{P_1,P_2,P_{3},P_4, P_{14},P_{15}\}$$
$$\mbox{ and }\;\; {\mathcal C}_8=\{P_1,P_2,P_{16}\}.$$
Notice that, taking into account our mentioned convention, in the proximity graph we have omitted the dotted curved edge corresponding with the proximity $P_5\rightarrow P_3$ because it can be deduced from the proximity $P_6\rightarrow P_3$. The points $P_1,P_2, P_3, P_4, P_{9}, P_{10}, P_{12},P_{14},P_{15}$ and $P_{16}$ are
free and the remaining ones satellite. In addition the curves $C_6,C_7$ and $C_8$ are nonsingular.


}
\end{exa}

\begin{figure}[htb]
\begin{center}
 \setlength{\unitlength}{0.8cm}
\begin{picture}(10,7)(0,0)
\put(5,0){\circle*{0.2}}
\put(5,0){\line(0,1){1}}
\put(4.5,-0.2){{\tiny 1}}
\put(5,1){\circle*{0.2}}
\put(5,1){\line(1,1){1}}
\put(5,1){\line(-1,1){1}}
\put(4.5,1){{\tiny 2}}
\qbezier[30](5,1)(7,1.5)(7,3)
\put(6,2){\circle*{0.2}}
\put(5.3,2){\tiny 16}
\put(4,2){\circle*{0.2}}
\put(6,2){\line(1,1){1}}
\put(3.8,1.5){\tiny 3}
\put(6,2){\vector(-1,1){1}}

\put(4.8,2.5){${\bf a}_8$}

\put(4,2){\line(-1,1){1}}

\put(6,3.2){${\bf a}_7$}

\put(7,3){\circle*{0.2}}
\put(6.5,3){\tiny 17}
\put(7,3){\vector(1,1){1}}
\put(8.1,4){${\bf a}_5$}
\put(3,3){\circle*{0.2}}
\put(2.5,3){\tiny 4}
\put(3,3){\line(1,0){1}}
\put(4,3){\circle*{0.2}}
\put(4,3){\line(1,1){1}}
\put(5,4){\circle*{0.2}}
\put(5,4){\vector(1,0){1}}
\put(5.1,4.1){\tiny 15}

\put(5,4){\vector(1,-1){1}}

\put(6,4.1){${\bf a}_6$}
\put(4.2,3){\tiny 14}
\put(3,3){\line(1,1){1}}
\put(4,4){\circle*{0.2}}
\put(3,3){\line(-1,1){1}}
\put(2,4){\circle*{0.2}}
\put(2.3,3.8){\tiny 5}
\put(4.3,4){\tiny 9}
\put(2,4){\line(-1,1){1}}
\put(1,5){\circle*{0.2}}
\put(1.2,5){\tiny 6}

\put(1,5){\line(-1,1){1}}
\put(0,6){\circle*{0.2}}
\put(0,6.2){\tiny 7}
\put(0,6){\vector(1,0){1}}
\put(1.2,6){${\bf a}_1$}

\qbezier[40](4,2)(1,3)(1,5)
\qbezier[40](2,4)(0,4)(0,6)

\put(4,4){\line(1,1){1}}
\put(5,5){\circle*{0.2}}
\put(5.3,5.1){\tiny 10}
\put(5,5){\line(1,1){1}}
\put(6.2,6){\tiny 11}
\put(6,6){\circle*{0,2}}
\put(6,6){\vector(-1,0){1}}
\put(4.3,6){${\bf a}_3$}
\qbezier[30](4,4)(4,5)(6,6)
\put(5,5){\line(1,0){2}}
\put(6,5){\circle*{0.2}}
\put(6.1,5.2){\tiny 12}
\put(7,5){\circle*{0.2}}
\put(7.2,5){\tiny 13}
\put(7,5){\vector(1,1){1}}
\put(8.2,6){${\bf a}_{4}$}
\qbezier[30](5,5)(6,4)(7,5)


\qbezier[30](3,3)(3.5,4)(3,5)

\put(2,4){\line(1,1){1}}
\put(3,5){\circle*{0.2}}
\put(2.5,5){\tiny 8}
\put(3,5){\vector(0,1){1}}
\put(3.2,6){${\bf a}_2$}


\end{picture}
\end{center}

\begin{center}
\setlength{\unitlength}{0.8cm}
\begin{picture}(12,6)(0,0)
\put(0,3){\line(1,0){12}}

\put(0,3){\circle*{0.2}}\put(1,3){\circle*{0.3}}\put(2,3){\circle*{0.2}}\put(3,2.9){\tiny $\blacksquare$}\put(3.9,2.9){$\bigstar$}\put(4.9,2.9){\tiny $\blacksquare$}\put(5.9,2.9){$\bigstar$}\put(7,3){\circle*{0.3}}\put(8,3){\circle*{0.2}}\put(8.9,2.9){$\bigstar$}\put(10,3){\circle*{0.2}}\put(10.9,2.9){$\bigstar$}\put(12,3){\circle*{0.2}}

\put(-0.2,3.2){\tiny 1}\put(0.8,3.2){\tiny 2}\put(1.8,3.2){\tiny 3}\put(2.8,3.2){\tiny 6}\put(3.8,3.2){\tiny 7}\put(4.8,3.2){\tiny 5}\put(5.8,3.2){\tiny 8}\put(6.8,3.2){\tiny 4}\put(7.8,3.2){\tiny 9}\put(8.8,3.2){\tiny 11}\put(9.8,3.2){\tiny 10}\put(10.8,3.2){\tiny 13}\put(11.8,3.2){\tiny 12}

\put(1,3){\line(1,1){2}}\put(1.8,3.9){$\bigstar$}\put(1.5,4){\tiny 17}
\put(3,5){\circle*{0.2}}\put(2.5,5){\tiny 16}\put(3,5){\vector(1,0){1}}\put(4.1,5){${\bf a}_8$}

\put(2,4){\vector(1,0){1}}\put(3.1,4){${\bf a}_5$}

\put(4,3){\vector(1,1){1}}\put(5.1,4){${\bf a}_1$}

\put(6,3){\vector(1,1){1}}\put(7.1,4){${\bf a}_2$}

\put(9,3){\vector(1,1){1}}\put(10.1,4){${\bf a}_3$}

\put(11,3){\vector(1,1){1}}\put(12.1,4){${\bf a}_4$}


\put(7,3){\line(-1,-1){2}}\put(6.2,2){\tiny 14}
\put(6,2){\circle*{0.2}} \put(5,1){\circle*{0.3}}\put(5.2,1){\tiny 15}

\put(5,1){\vector(-1,-1){1}}
\put(5,1){\vector(1,-1){1}}

\put(4.3,0){${\bf a}_6$}
\put(6.2,0){${\bf a}_7$}

\end{picture}
\end{center}
\caption{Proximity and dual graphs of Example
\ref{example1}.}
\end{figure}
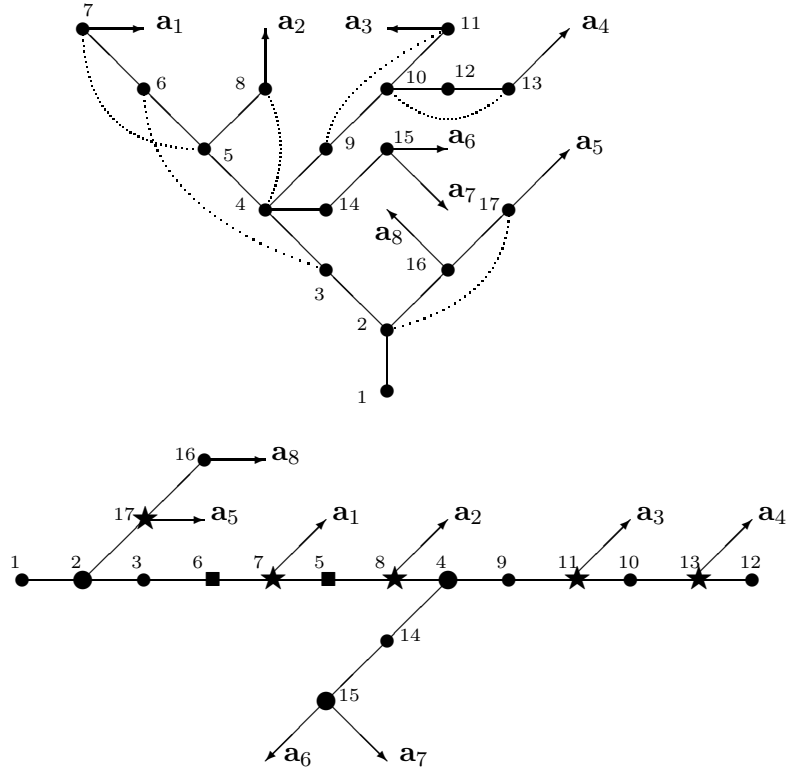

\section{Results}
\label{results}

In this section we keep the notation given in the previous one. Let $h$
be an irreducible element in $R$ such that $\pi$ is a log resolution
of the curve $H$ that it defines. Let
$\{\betabarra_l^{h}\}_{l=0}^g$ be its maximal contact values (see
Section 2 of Chapter IV in \cite{Campillo}, for instance). This set
of values is an increasing sequence of positive integers which is a
minimal generating set of the so-called semigroup of values of $H$. In
addition, they constitute an equivalent datum to the embedded
topological type of $H$ and, in the complex case, can be easily
computed from the set of Puiseux pairs given by a primitive
parametrization of the curve. When $H$ is nonsingular, the unique
defined maximal contact value is $\betabarra_0^{h}=1$. In this paper
and in this case, we also define $\betabarra_1^{h}$ as the number of
points of $\mathcal C$ through which the strict transforms of $H$
pass.

\begin{de}
{\rm Let $h$ and $H$ be as above. A \emph{terminal satellite point
for $h$ (or for $H$)} is a point $P_j\in {\mathcal C}$ such that it is
satellite and  the set
$$\left\{P_k\in {\mathcal C}\setminus \{P_j\}\mid\mbox{ the strict
transform of $H$ on $X_k$ passes through $P_k$ and } P_k\gtrsim
P_j \right\}$$ is either empty or its minimum (with respect to the
ordering $\gtrsim$) is a free point.}
\end{de}

Notice that the number of terminal satellite points for $h$
coincides with the index $g$ of the last maximal contact value
$\betabarra_g^{h}$ and that the part of the dual graph of $H$ up
to the $l$th terminal satellite point, $1 \leq l \leq g$, can be deduced from the set
of values $\{\betabarra_0^{h}, \betabarra_1^{h},  \ldots,
\betabarra_l^{h}\}$. When the curve $H$ is not singular, we denote
by $l_0^h$ the cardinality of the set of free points $P_j \in
{\mathcal C}$ through which the strict transforms of $H$ pass.
Otherwise, $l_0^h + 1$ will stand for the cardinality of the set of free points $P_j \in
{\mathcal C}$ satisfying the above condition and such that every terminal
satellite point for $h$  is infinitely near to $P_j$ (see
\cite{Delgado} for the source of the notation); in other words, $l_0^h+1$ is the length $\ell$ of the maximal chain of initial consecutive free points $P_1\lesssim \ldots \lesssim P_{\ell}$ such that the strict transforms of $H$ pass through $P_i$ for all $i\in \{1,\ldots,\ell\}$. When $h=f_i$ for
some $i\in \{1, 2, \ldots,r\}$, we have set $C_i$ instead of $F_i$ and, for simplicity, we will write
$\betabarra_k^i$ (respectively, $l_0^i$) instead of
$\betabarra_k^{f_i}$ (respectively, $l_0^{f_i})$.

Assume from now on that the curves $C_1,\ldots,C_n$ are singular
and the remaining ones, $C_{n+1},\ldots,C_r$, are nonsingular. For each
$i\in \{1, 2, \ldots,n\}$,  we denote by $P_{t_i}$ the minimum
terminal satellite point for $C_i$ with respect to the ordering
$\lesssim$.
The following two sets will be useful:
$${\mathcal T}:=\{P_{t_i}\mid i=1,\ldots,n\}$$
and
$${\mathcal F}:=\left\{P_j\in {\mathcal C} \mid  P_j \lesssim P_{t_i}
\mbox{ for some } P_{t_i}\in {\mathcal T}\right\}
\cup {\mathcal C}_{n+1}\cdots \cup {\mathcal C}_{r}.$$

Let us give an example with the aim of clarifying previous concepts.
\begin{exa}
{\rm

Consider again the curve in Example \ref{example1}. Then $n=5$ and, from the proximity graph, one can see that
$${\mathcal T}=\{P_{t_1}=P_7, P_{t_2}=P_{8}, P_{t_3}=P_{11}, P_{t_4}=P_{13}, P_{t_5}=P_{17}\}.$$
In the dual graph depicted in Figure 1, we have labeled those vertices representing divisors $E_j$ such that $P_j$ is a satellite point and it is not in $\mathcal T$ (respectively, it is in $\mathcal T$) with a square (respectively, a star).

Notice that, in this example, ${\mathcal F}={\mathcal
C}$.

The free points of ${\mathcal C}_1$ such that the unique terminal
satellite point in ${\mathcal C}_1$ is proximate to them (that is, the initial free
points in the branch of the proximity graph corresponding to
${\mathcal C}_3$) are $P_1, P_2,P_3$ and $P_4$. Therefore
$l_0^1+1=4$.
Similarly $l_0^2=3, l_0^3=5, l_0^4=6, l_0^5=2, l_0^6=l_0^7=6$ and $l_0^8=3$.

}
\end{exa}

Next, we will state our first result which, in the complex case, can
be deduced from the proof of \cite[Th. 1.3]{Kuwata}. It shows that
the log-canonical threshold of $C$ depends only on the part of the
dual graph determined by the values $\betabarra_0^i$ and $\betabarra_1^i$
of every component $C_i$ of $C$.
\begin{pro}\label{kuwatilla}
The log-canonical threshold of $C$ is the minimum of the elements
in the set
\[\left\{{\overline \alpha}_j =\frac{a_j+1}{b_j}\mid P_j\in{\mathcal F}\right\}.
\]
\end{pro}

 Our main result uses a weight function for some vertices of $\Gamma(C)$. Before stating it, we introduce these weights and some necessary concepts.

\begin{de}
{\rm Let $h_1$ and $h_2$ be two irreducible elements in $R$ such
that $\pi$ is a common log resolution of the curves $H_1$ and $H_2$
($\neq H_1$) that they define. The \emph{contact pair} of $h_1$ and
$h_2$, $(h_1\mid h_2)$, is defined to be the couple of integers
$(q,c)$ such that: \begin{itemize} \item $q$ is the number of common
terminal satellite points for $h_1$ and $h_2$.
\item $c$ is the cardinality of the set of free points $P_j \in \mathcal{C}$ such that the strict transforms of $H_1$ and $H_2$ pass through $P_j$ and $P_j$ is infinitely near to
the last common terminal satellite point (if any).

\end{itemize} }
\end{de}

An equivalent definition of  contact pair using the
Hamburger-Noether expansion of the branches is given in
\cite{Delgado}.

\begin{exa}
{\rm Let us return to Example \ref{example1}. $C_1$ and $C_3$ do
not share any terminal satellite point and, moreover, the free points in ${\mathcal C}_1 \cap {\mathcal
C}_3$ are $P_1, P_2, P_3$ and $P_4$; therefore $(f_1\mid
f_3)=(0,4)$.

Now, suppose that $h_1$ and $h_2$ are irreducible elements
in $R$ such that the minimal log resolution of the curve defined by $h_1$ (respectively, $h_2$) is given by blowing-up at the points $P_1, P_2, P_{16}, P_{17}$
and $P_{18}, P_{19}$ (respectively, and $P_{20}, P_{21}$). $P_{18}, P_{19}, P_{20}, P_{21}$ are new points enlarging $\mathcal{C}$ up to a new configuration $\mathcal{C}'$
such that all of them are proximate to $P_{17}$, $P_{18}$ and $P_{20}$ are two distinct points in the first infinitesimal neighborhood of $P_{17}$, $P_{19} \gtrsim P_{18}$ and $P_{21} \gtrsim P_{20}$. Set also $\pi$ the log resolution given by $\mathcal{C}'$. Then, each curve $H_1$ and $H_2$ has two terminal satellite points and $(h_1|h_2)=(1,0)$ because $P_{17}$ is their last common infinitely near point.

}
\end{exa}

\begin{de}
{\rm Two components ${C}_{i_1}$ and ${C}_{i_2}$ of the curve $C$
are called to be  \emph{separated} at a point $P_j\in {\mathcal
C}$ when $\max_{\gtrsim} ({\mathcal C}_{i_1}\cap {\mathcal C}_{i_2})
= P_j$.

Moreover, ${C}_{i_1}$ and ${C}_{i_2}$ will be \emph{freely
separated} at $P_j$ if they are separated at $P_j$ and $(f_{i_1}\mid
f_{i_2}) = (0,c)$ for some $c\leq \min\{l_0^{{i_1}},l_0^{{i_2}}\}$.

Finally, a point $P_j\in {\mathcal C}$ is an \emph{initial
separating point} if at least two components ${C}_{i_1}$ and
${C}_{i_2}$ of $C$ are freely separated at $P_j$.}
\end{de}

\begin{rem}
{\rm
Notice that if two  components of $C$  are freely separated at a point $P_j$, then $P_j$ must be a free point.
}
\end{rem}

The above defined set $\mathcal{T}$ and the set
 $$\mathcal{S} := \{ \mbox{initial separating
points of $\mathcal C$} \}$$ will be useful for our purposes. Notice that the points in $\mathcal T$  (respectively, $\mathcal S$) are satellite (respectively, free) and, therefore, $\mathcal T\cap \mathcal S=\emptyset$. Moreover ${\mathcal T}\cup {\mathcal S}\subseteq {\mathcal F}$.

\begin{exa}
\label{examp4}
{\rm In Example \ref{example1} and from the proximity graph one can deduce the following:

$C_1$ and $C_2$ are separated at $P_5$, that is a satellite point; therefore $C_1$ and $C_2$ are not freely separated.

$C_1$ and $C_3$ (respectively, $C_4$, $C_6$, $C_7$) are separated at $P_4$, that is a free point; however they are not freely separated because $(f_1\mid f_i)=(0\mid 4)$ for $i\in\{3,4,6,7\}$ and $4=l_0^1+1$.

$C_1$ and $C_5$ (respectively, $C_8$) are freely separated at $P_2$ because $(f_1\mid f_5)=(0,2)$ (respectively, $(f_1\mid f_8)=(0,2)$) and $2\leq \min\{l_0^1=3,l_0^5=2\}$ (respectively, $2\leq \min\{l_0^1=3,l_0^8=3\}$).

$C_3$ and $C_4$ are separated at $P_{10}$, that is a free point, but they are not freely separated because $(f_3\mid f_4)=(0,6)$ and $6=l_0^3+1$.

Analogously, $C_3$ (respectively, $C_4$) and $C_6$ are freely separated at $P_4$. $C_3$ (respectively, $C_4$) and $C_7$ are also freely separated at $P_4$. $C_6$ and $C_7$ are freely separated at $P_{15}$.
$C_8$ and $C_5$ are not freely separated at $P_{16}$ because $(f_8\mid f_5)=(0,3)$ and $3=l_0^5+1$.
$C_1$ (respectively, $C_2$, $C_3$, $C_4$, $C_6$, $C_7$) and $C_5$ are freely separated at $P_2$. The same happens taking $C_8$ instead of $C_5$.

Finally, we add that the set of initial separating points of the curve $C$ in this example is
${\mathcal S}=\{P_2,P_4, P_{15}\}$. These points have been marked with thicker dots in the dual graph depicted in Figure 1. }
\end{exa}


Returning to our development, where $C$ is a reduced curve with $r$ branches, let ${\mathcal V}_{\mathcal F}$ be the set of vertices ${\bf v}_j$ of $\Gamma (C)$ such that $P_j\in {\mathcal F}$ and consider the following subsets of ${\mathcal V}_{\mathcal F}$:
$${\mathcal V}_{\mathcal T}:= \{ \mathbf{v}_j | P_j
\in {\mathcal T} \},\;\;{\mathcal V}_{\mathcal S}:= \{ \mathbf{v}_j |
P_j \in {\mathcal S} \},\;\;
{\mathcal V}_{\mathrm{free}}:= \{ \mathbf{v}_j\in {\mathcal V}_{\mathcal F} |
P_j \mbox{ is a free point}\},$$
$$
{\mathcal V}_{\mathrm{end}}:= \{\mathbf{v}_j \in {\mathcal V}_{\mathcal F} \; | \; \mbox{ $\mathbf{v}_j$ has  degree 1}\}\;\;\mbox{ and}\;\;{\mathcal V}:={\mathcal V}_{\mathcal T}\cup {\mathcal V}_{\mathcal
S}.
$$
${\mathcal V}_{\mathrm{end}}$ is a subset of the set formed with the vertex $\bv_1$ and those vertices $\bv_j$ such that $P_j$ is a maximal free point of ${\mathcal F}$ (with respect to the ordering $\lesssim$).
 Notice that the set ${\mathcal V}$ can be easily located in the dual graph because it is the subset of vertices in ${\mathcal V}_{\mathcal F}$ with \emph{adapted degree} greater than or equal to 3, where the \emph{adapted degree} of a vertex ${\bf v}_j\not={\bf v}_1$ (respectively, ${\bf v}_1$) is defined as the sum of its degree and the number of arrows labeling it (respectively, one plus the number we have just defined).

As usual, a \emph{path} $\gamma$ in $\Gamma(C)$ is an
ordered sequence of different vertices  of $\Gamma(C)$ such that
two consecutive ones are joined by an (oriented) edge. If $a$ is
the initial vertex, in$(\gamma)$, and $b$ the terminal vertex (which can also be given by an arrow, if any, that labels it), ter$(\gamma)$,
we will denote $\gamma$ by $[a,b]$ and also, by an abuse of notation, its set of vertices. Moreover, we will use $]a,b]$ (respectively, $]a,b[$) to denote the path (or its set of vertices) obtained after removing from $\gamma$ the initial vertex (respectively, the initial and final vertices) and its incident edge (respectively, their incident edges).


Let $\leq$ denote the order induced by the (oriented) tree $\Gamma(C)$ in its set of vertices, that is, given two vertices ${\bf v}_{j_1}$ and ${\bf v}_{j_2}$, ${\bf v}_{j_1}\leq {\bf v}_{j_2}$ means that ${\bf v}_{j_1}$ belongs to $[{\bf v}_1,{\bf v}_{j_2}]$. By convention, if ${\bf a}_i$ is an arrow that is a label of ${\bf v}_{j_2}$ then ${\bf v}_{j_1}\leq {\bf a}_i$ will mean ${\bf v}_{j_1}\leq {\bf v}_{j_2}$.

Also, for every vertex ${\bf v}_j\in {\mathcal V}_{\mathcal F}$, we define
\begin{equation}
\label{order+}
{\bf v}_j^{<}:=\{{\bf a}_i \mid {\bf v}_j \nleq
{\bf a}_i\}\;\;\;\mbox{and}\;\;\; {\bf v}_j^{\geq}:=\{{\bf a}_i
\mid {\bf v}_j \leq {\bf a}_i\}.
\end{equation}
Furthermore we consider the map $\sigma: {\mathcal V}_{\mathcal F}\rightarrow \mathbb{Z}$
given by
\begin{equation}
\label{sigma}
\sigma({\bf v}_j)=\sum_{{\bf a}_{i}\in {\bf v}_j^{<}} c_{ji}
\betabarra_0^i-\sum_{{\bf a}_{i}\in {\bf v}_j^{\geq }}
\betabarra_0^i,
\end{equation}
where
\begin{equation*}
c_{ji}:= \left \{ \begin{array}{ll} \text{card} \left([{\bf v}_1,
{\bf a}_i] \cap [{\bf v}_1, {\bf v}_{j}]\cap {\mathcal V}_{\mathrm{free}}\right)  & \text{if} \; \;
\text{ter}\left( [{\bf v}_1, {\bf a}_i] \cap [{\bf v}_1,
{\bf v}_{j}]\right) \in \mathcal{S}, \\
  \betabarra_1^i/\betabarra_0^i & \text{otherwise}.
\end{array} \right.
\end{equation*} \vspace{1mm}

The following definition will be useful in the next section and, moreover, it will allow us to clarify the meaning of the coefficients $c_{ij}$ and the condition $\text{ter}\left( [{\bf v}_1, {\bf a}_i] \cap [{\bf v}_1,
{\bf v}_{j}]\right) \in \mathcal{S}$ appearing in the above formula (see Remark \ref{nota2}).

\begin{de}
{\rm Consider a point $P_j\in {\mathcal C}$, ${\mathcal C}$ being
the configuration of the mentioned resolution $\pi$ of $C$. A
\emph{curvette} at $P_j$ will be an irreducible element $\varphi\in
R$ defining a  divisor on ${\rm Spec}(R)$ different from $C_i$ for
all $i\in \{1, 2, \ldots,r\}$,  whose strict transform on $X$ is not
singular and meets $E_j$ transversally at a regular point.
}
\end{de}

\begin{rem}\label{nota2}
{\rm
Notice that, if $\bv_j\in {\mathcal V}_{\mathcal F}$ and ${\bf a}_{i}\in {\bf v}_j^{<}$, then
the condition $\text{ter}\left( [{\bf v}_1, {\bf a}_i] \cap [{\bf v}_1,
{\bf v}_{j}]\right) \in \mathcal{S}$ is equivalent to say that a curvette at $P_j$, $\psi$, and $C_i$ are freely separated (as components of a reduced curve that contains both branches). Moreover, in this case, the integer $c_{ji}$ appearing in the definition of the map $\sigma$ coincides with the integer $c$ such that $(\psi \mid f_i)=(0,c)$.
}
\end{rem}

\begin{rem}\label{nota}
{\rm
The sets ${\bf v}_j^{<}$ and ${\bf v}_j^{\geq}$ do not change when ${\bf v}_j$
 runs over a path $]{\bf v}_{j_1},{\bf v}_{j_2}]$ such that ${\bf v}_{j_1},{\bf v}_{j_2}\in {\mathcal V}\cup {\mathcal V}_{\mathrm{end}}$ and $]{\bf v}_{j_1},{\bf v}_{j_2}[\cap {\mathcal V}=\emptyset$. Therefore the map $\sigma$ is constant in such a path.
}
\end{rem}

Let us denote by $I(f_i,f_s)$ the intersection multiplicity of two
different components $C_i$ and $C_s$, $0 \leq i,s \leq r$, of $C$.
We are ready to state our main result, which is:

\begin{teo}\label{gordo2}
Let $C$ be a singular reduced (germ of) plane curve with $r$ branches and $\mathcal{C}$
(respectively, $\Gamma (C)$) the constellation of infinitely near
points (respectively, dual graph) associated with its minimal log resolution. Consider the subsets ${\mathcal
T}$ and ${\mathcal S}$ of  $\; \mathcal{C}$ and the subset of
vertices ${\mathcal V}={\mathcal V}_{\mathcal T}\cup {\mathcal V}_{\mathcal S}$ of $\; \Gamma (C)$ above defined. Then:

\textbf{(1)} There exists a vertex ${\bf v}_k\in {\mathcal V}$
satisfying the conditions:
\begin{itemize}
\item[(\textit{a})] $\sigma ({\bf v}_j)< 0$ for all ${\bf v}_j\in [{\bf v}_1, {\bf v}_{k}]\cap {\mathcal V}$ and
\item[(\textit{b})] $\sigma ({\bf v}_j)\geq 0$ for all ${\bf v}_j\in {\mathcal V}\setminus [{\bf v}_1, {\bf v}_{k}]$.
\end{itemize}

 \textbf{(2)} The log-canonical threshold of $C$ is the value ${\overline \alpha}_{k}$ defined in the above equality (\ref{WW}) and it can be computed as follows:
\begin{itemize}
\item If ${\bf v}_k={\bf v}_{t_i}\in {\mathcal V}_{\mathcal T}$, then
$${\overline \alpha}_k={\overline \alpha}_{t_i}=
\frac{\betabarra_0^i+\betabarra_1^i}{\sum_{s=1}^r \delta_{is}},$$
where
\begin{equation*}
\delta_{is}= \left \{ \begin{array}{ll}
\betabarra_0^i\betabarra_1^s
& \mbox{if either $s=i$, or $s\not=i$ and
$\betabarra_0^i\betabarra_1^s=\betabarra_1^i\betabarra_0^s\leq
I(f_i,f_s)$},\\
I(f_i,f_s)
& \text{otherwise}.
\end{array}
\right.
\end{equation*}

\item If ${\bf v}_k\in  {\mathcal V}_{\mathcal S}$, then
$${\overline \alpha}_k=\frac{\betabarra_0^{i_1}
\betabarra_0^{i_2}+I(f_{i_1},f_{i_2})}{\betabarra_0^{i_1}I(f_{i_1},f_{i_2})
+\betabarra_0^{i_2} \sum_{1 \leq s \leq r,\; s \neq i_1} I(f_{i_1},f_{s})},$$ where
$C_{i_1}$ and $C_{i_2}$ are any two components which are freely separated
at $P_{k}$.

\end{itemize}

\end{teo}

\begin{rem}
{\rm

Notice that the vertex ${\bf v}_{k}$ mentioned in Theorem \ref{gordo2} can be easily obtained applying, to the graph $\Gamma(C)$, an obvious variant of the well-known breadth-first search strategy used in  graph theory.

}
\end{rem}

\begin{exa}\label{example}
{\rm

Consider the curve in Example \ref{example1} again. The sets of
maximal contact values of the components $C_1, C_2, \ldots,C_5$
are, respectively, $\{\betabarra_0^1,\betabarra_1^1\}=(5,17\}$,
$\{\betabarra_0^2,\betabarra_1^2\}=\{3,11\}$,
$\{\betabarra_0^3,\betabarra_1^3\}=\{2,11\}$,
$\{\betabarra_0^4,\betabarra_1^4\}=\{2,13\}$ and $\{\betabarra_0^5,\betabarra_1^5\}=\{2,5\}$.  We have that
$${\mathcal V}={\mathcal V}_{\mathcal S}\cup {\mathcal V}_{\mathcal T}=\{{\bf v}_2,{\bf v}_4,{\bf v}_{7},{\bf v}_8,{\bf v}_{11},{\bf v}_{13}, {\bf v}_{15},  {\bf v}_{17}\}.$$
Moreover
$$\sigma({\bf v}_2)=-\sum_{i=1}^8\betabarra_0^i=-17,\;\;\;\;\sigma({\bf v}_7)=2\betabarra_0^5+2\betabarra_0^8-\betabarra_0^1-\betabarra_0^2-\betabarra_0^3-\betabarra_0^4-\betabarra_0^6-\betabarra_0^7=-4$$ $$ \mbox{and } \sigma({\bf v}_8)=2\betabarra_0^5+2\betabarra_0^8+\betabarra_1^1-\betabarra_0^2-\betabarra_0^3-\betabarra_0^4-\betabarra_0^6-\betabarra_0^7=14.$$
Therefore, ${\bf v}_7$ is the distinguished vertex ${\bf v}_k$ in Theorem \ref{gordo2}. Then
$${\rm lct}(C)={\overline \alpha}_7={\overline \alpha}_{t_1}=\frac{\betabarra_1^1+
\betabarra_0^1}{\betabarra_1^1\betabarra_0^1+\sum_{s=2}^8
I(f_1,f_s)} = $$
$$=\frac{17+5}{17\cdot 5+17\cdot 3+17\cdot 2+17\cdot 2+ 2\cdot 5\cdot 2+17\cdot 1+17\cdot 1+2\cdot 5\cdot 1}=\frac{11}{134}.$$

}
\end{exa}

\begin{rem}
{\rm
Notice that our  example might induce the reader to think that one has to consider a large number of vertices of $\Gamma (C)$. However, generally speaking, this is not true since we only need those vertices in $\mathcal{V} \subseteq \mathcal{V}_{\mathcal{F}}$,  this last set being (proportionally) very small  when $C$ has branches with many contact maximal values. We have considered a case where rather vertices are relevant to illustrate a wide spectrum of possibilities but avoiding unnecessary information.
}
\end{rem}

As a consequence of Theorem \ref{gordo2} and the forthcoming Lemma \ref{torero}, we state the following result which determines the exact value of the log-canonical threshold of a reduced curve with two branches.

\begin{cor}
\label{corol}
 Assume that the number of components of $C$ is $r=2$
and, without loss of generality, that
$\betabarra_1^1/\betabarra_0^1\leq \betabarra_1^2/\betabarra_0^2$.
Then:
\begin{itemize}
\item[(a)] If $C_1$ and $C_2$ are not freely separated, it holds
that
\begin{equation*}
\mathrm{lct}(C)= \left \{ \begin{array}{ll}
\frac{\betabarra_1^1+\betabarra_0^1}{\betabarra_1^1(\betabarra_0^1
+\betabarra_0^2)}& \text{if} \; \; \betabarra_1^1\geq \betabarra_0^2,\\
\frac{\betabarra_1^2+\betabarra_0^2}{\betabarra_0^2(\betabarra_1^1+\betabarra_1^2)}
& \text{otherwise}.
\end{array}
\right.
\end{equation*}

\vspace{1mm}

\item[(b)] If, on the contrary, $C_1$ and $C_2$ are freely
separated,
\begin{equation*} \mathrm{lct}(C)= \left \{
\begin{array}{ll}
\frac{\betabarra_0^1\betabarra_0^2+I(f_1,f_2)}{(\betabarra_0^1+
\betabarra_0^2)I(f_1,f_2)} & \text{if} \; \; \frac{1}{c}\leq
\frac{\betabarra_0^2}{\betabarra_0^1}\leq c,\\
 \frac{\betabarra_1^1+\betabarra_0^1}{\betabarra_0^1 \betabarra_1^1 + I(f_1,f_2)}&  \text{if} \; \;
\frac{\betabarra_0^2}{\betabarra_0^1}<\frac{1}{c}, \\
\frac{\betabarra_1^2+\betabarra_0^2}{\betabarra_0^2 \betabarra_1^2 + I(f_1,f_2)} & \text{otherwise}, \\
\end{array}
\right.
\end{equation*}
$c$ being the integer such that $(f_1\mid f_2)=(0,c)$.
\end{itemize}

\end{cor}

We finish this section with two remarks concerning the log-canonical threshold of a non-reduced plane curve and of a complete ideal of our ring $R$.

\begin{rem} \label{rem5}
{\rm Assume that $C=\sum_{i=1}^r n_i C_i$ is a non-reduced curve, $C_1, C_2, \ldots, C_r$ being its integral components. Let ${\mathcal C}$ be the configuration of infinitely near points associated with a log-resolution of $C$. For each $i\in \{1, 2, \ldots,r\}$,  pick $n_i$ curves   $C_{i1}, C_{i2}, \ldots,C_{in_i}$ defined by  $n_i$ curvettes at the maximal point $P_{j-1}$ of $\mathcal C$ through which the strict transform of $C_i$ passes and such that their strict transforms meet $E_j$ at different free points. Then, it follows from the definition of log-canonical threshold that $${\rm lct}(C)=\min
\{1/n_1,\ 1/n_2, \ldots, 1/n_r,{\rm lct}(C')\},$$ where $C'$ is the reduced curve
$\sum_{1\leq i\leq r}\sum_{1\leq j\leq n_i} C_{ij}$. Therefore Theorem \ref{gordo2} provides, in fact, a formula for the log-canonical threshold of any (reduced or non-reduced) singular plane curve.
}
\end{rem}

\begin{rem} \label{rem6}
{\rm Theorem \ref{gordo2} also provides, as a byproduct, a formula for the log-canonical threshold of a  complete ideal of finite co-length in $R$. Indeed, if $\mathfrak{a}$ is such an ideal, it has a unique factorization $\mathfrak{a}=\mathfrak{p}_1^{n_1}\cdots \mathfrak{p}_r^{n_r}$ as a product of simple complete ideals  \cite[page 385]{z-s}. Then, it is straightforward from the definition that ${\rm lct}(\mathfrak{a})={\rm lct}(\sum_{i=1}^r D_i)$ where, for each $i=1,\ldots,r$, $D_i$ is a sum of $n_i$ suitable chosen \emph{general} curves of the ideal $\mathfrak{p}_i$. Recall that a \emph{general} curve of a simple complete ideal $\mathfrak{p}_i$ is an irreducible curve whose strict transform, on the surface  given by the point blowing-up sequence providing the exceptional divisor $E_i$ that defines the ideal, meets $E_i$ transversally at a nonsingular point. Also, notice that in the previous sentence ``suitable chosen" means that the curves meet $E_i$ at different points.
}
\end{rem}

\section{Proofs}
\label{proofs}

This section is devoted to prove the results that we have stated in the
previous one. To this purpose, in each subsection, we will
introduce some notation and prove some properties which
will be necessary to deduce our results. Notation and lemmas in
Subsection \ref{31} are also useful for Subsection \ref{32}.

\subsection{Proposition \ref{kuwatilla}: auxiliary results and proof}
\label{31}

We start this section with a lemma which is deduced from \cite[Section 3]{Delgado} and
will be a key tool for our proofs.

\begin{lem}\label{torero}
Let $h_1$ and $h_2$ be two irreducible elements of $R$ such that
$\pi$ is a common log resolution of the curves $H_1$ and $H_2$ ($H_1
\neq H_2$) that they define. Set $I(h_1,h_2)$ (respectively, $(h_1 \mid
h_2)=(q,c)$) the  intersection multiplicity (respectively, contact pair) of $h_1$ and $h_2$. Then:
\begin{itemize}
\item[(a)] $q\geq 1$ if, and only if,
$\betabarra_0^{h_1}\betabarra_1^{h_2}=\betabarra_1^{h_1}\betabarra_0^{h_2}\leq
I(h_1, h_2).$
\item[(b)] If $q=1$ and $c=0$,  then $I(h_1,h_2)=\betabarra_0^{h_1}\betabarra_1^{h_2}=\betabarra_1^{h_1}\betabarra_0^{h_2}$.
\item[(c)] If $q=0$ and $c\leq
\min\{l_0^{h_1},l_0^{h_2}\}$,  then $I(h_1,
h_2)=c\betabarra_0^{h_1}\betabarra_0^{h_2}$. \item[(d)] If $q=0$ and
$c=\min\{l_0^{h_1}+1,l_0^{h_2}+1\}$,  then $I(h_1, h_2)=\min
\{\betabarra_0^{h_1}\betabarra_1^{h_2},
\betabarra_1^{h_1}\betabarra_0^{h_2}\}$.

\end{itemize}

\end{lem}

With the previous notation, for every curvette $\varphi$ at a point of ${\mathcal F}$, we consider the following subsets of $\mathbb{J}_r := \{1, 2,\ldots,r\}$:
\begin{itemize}
\item $J_1(\varphi) := \left\{ s \in \mathbb{J}_r \;|\; (\varphi \mid
f_s)=(0,c) \mbox{ with } c=\min \{l^{\varphi}_0+1,l^s_0+1\} \mbox{ and }
\frac{\betabarra_1^s}{\betabarra_0^s}<
\frac{\betabarra_1^{\varphi}}{\betabarra_0^{\varphi}} \right\}$,
 \item $J_2(\varphi) :=
\left\{ s\in \mathbb{J}_r \;|\; ({\varphi}\mid f_s)=(0,c) \mbox{ with
} c=\min \{l^{\varphi}_0+1,l^s_0+1\} \mbox{ and }
\frac{\betabarra_1^{s}}{\betabarra_0^{s}}>
\frac{\betabarra_1^{\varphi}}{\betabarra_0^{\varphi}} \right\}$,
  \item $J_3(\varphi) := \left\{ s\in  \mathbb{J}_r \;\mid \;(\varphi\mid f_s)=(1,c)
  \mbox{ for some } c\geq 0  \right\}$,
    \item $J_{4}(\varphi)  := \left\{ s \in \mathbb{J}_r \;|\;
  ({\varphi}\mid f_s)=(0,c)
   \mbox{ with } c\leq\min\{l^{\varphi}_0, l^s_0\}\right\}$.
 \end{itemize}
Notice that the non-empty elements of $\{J_1(\varphi),J_2(\varphi),J_3(\varphi),J_{4}(\varphi)\}$ define a partition of $\mathbb{J}_r$. If $\varphi$ defines a nonsingular curve then $J_4(\varphi)=J_{4,1}(\varphi)\cup J_{4,2}(\varphi)$, where
\begin{itemize}
\item $J_{4,1}(\varphi)  :=  \left\{ s \in \mathbb{J}_r \;|\;
({\varphi}\mid f_s)=(0,c)
 \mbox{ with } c <l^{\varphi}_0 \mbox{ and } c\leq l^s_0\right\}$ and
\item $J_{4,2}(\varphi)  :=  \left\{ s \in \mathbb{J}_r \;|\;
({\varphi}\mid f_s)=(0,c)
 \mbox{ with } c =l^{\varphi}_0 \leq l^s_0\right\}$.

\end{itemize}

The following lemmas provide some properties that use and study the sets $J_{l}(\varphi)$. They will be needed for the development of our paper.

\begin{lem}\label{anthrax}
Let $\varphi$ be a curvette at a free point $P_j\in {\mathcal F}$. Then $\varphi$ defines a nonsingular curve and $J_2(\varphi)=J_{3}(\varphi)=\emptyset$.
\end{lem}

\begin{proof}

Since $P_j$ is a free point that belongs to $\mathcal F$ and $\varphi$ is transversal to $E_j$, $\varphi$ defines a nonsingular curve. Then, the equality $J_{3}(\varphi)=\emptyset$ is clear because there is no terminal satellite point for $\varphi$.

Reasoning by contradiction, assume that there exists $s\in \{1,\ldots,r\}$ such that $s\in J_{2}(\varphi)$. Then $C_s$ is a singular curve and, taking into account that $\varphi$ is nonsingular, we have that $(\varphi\mid f_s)=(0,c)$ with $c=l_0^s+1$ and $\betabarra_1^s/\betabarra_0^s>c$. That is, $c$ is the number of free points in ${\mathcal C}_s\cap {\mathcal F}$ and $\betabarra_1^s >c\betabarra_0^s$. But this contradicts Noether's formula because $\betabarra_0^s$ is the multiplicity of $C_s$, $\betabarra_1^s=I(\varphi,f_s)$ and $C_s$ is singular.

\end{proof}

\begin{lem}\label{ttt}

With the above notations, let $P_j\in {\mathcal F}$ and let $\varphi$ be a curvette at $P_j$.
\begin{itemize}
\item[(a)] If $P_j=P_{t_i}\in {\mathcal T}$, then
$${\bf
v}_{t_i}^{<}=\{{\bf a}_s\mid s\in J_1(\varphi)\cup J_4(\varphi)\}\;\;\mbox{ and }\;\;{\bf
v}_{t_i}^{\geq}=\{{\bf a}_s\mid s\in J_2(\varphi)\cup J_3(\varphi)\}.$$
\item[(b)] If $P_j$ is free, then $${\bf
v}_{j}^{<}=\{{\bf a}_s\mid s\in J_1(\varphi)\cup J_{4,1}(\varphi)\}\;\;\mbox{ and }\;\;{\bf
v}_{j}^{\geq}=\{{\bf a}_s\mid s\in J_{4,2}(\varphi)\}.$$
\end{itemize}

\end{lem}

\begin{proof}

For any couple of irreducible elements $h_1,h_2\in R$ $(H_1 \neq
H_2$), we define
$$H(h_1,h_2):=\frac{I(h_1,h_2)}{\betabarra_0^{h_2}}.$$

In order to prove (a), first consider an
index $s\in J_1(\varphi)$ and let $\varphi'$ be a curvette at $P_j$ different from $\varphi$. Lemma
\ref{torero} implies that
\begin{equation}\label{prrr}
H(\varphi,f_s)=\frac{\betabarra_1^s\betabarra_0^{\varphi}}{\betabarra_0^s}\;\;\;\mbox{
and }\;\;\;H(\varphi,\varphi')= \betabarra_1^{\varphi}.
\end{equation}
Since $\betabarra_1^{s}/\betabarra_0^{s}<\betabarra_1^{\varphi}/\betabarra_0^{\varphi}$,  the following inequality holds:
\begin{equation}\label{fff}
H(\varphi,f_s)<H(\varphi,\varphi').
\end{equation}

By a result  stated in \cite[page 425]{Felix}, and proved in \cite{Delgado2}, the last inequality is equivalent to fact that if ${\bf v}_k$ is the vertex satisfying $[{\bf v}_1,{\bf v}_j]\cap [{\bf v}_1,{\bf a}_s]=[{\bf v}_1,{\bf v}_k]$ then ${\bf v}_k<{\bf v}_j$. Moreover, since $(\varphi\mid f_s)=(0,c)$ with $c=\min\{l_0^{\varphi}+1,l_0^{s}+1\}$, the process of construction of the dual graph $\Gamma(C)$ allows us to deduce that either ${\bf v}_k={\bf v}_j$ or ${\bf v}_k= {\bf a}_s$. So, the unique possibility is ${\bf v}_k= {\bf a}_s$ and, therefore, ${\bf a}_s\in {\bf v}_j^<$. We notice that, although the mentioned result of \cite{Felix, Delgado2} is stated over the complex numbers, its proof depends only on the Hamburger-Noether expansions of the curves, which are independent of the characteristic of the ground field.

Assume now that $s \in J_2(\varphi)$ and consider a curvette $\psi$ at the point $P_r$ such that ${\bf a}_s$ is a label of ${\bf v}_r$. Then, the inequality
\begin{equation}\label{ggg}
H(f_s,\varphi)<H(f_s,\psi)
\end{equation}
holds because $H(f_s,\varphi)=\betabarra_1^{\varphi}\betabarra_0^{s}/\betabarra_0^{\varphi}$,  $H(f_s,\psi)=\betabarra_1^s$ and $\betabarra_1^{\varphi}/\betabarra_0^{\varphi}<\betabarra_1^{s}/\betabarra_0^{s}$. Also, again by \cite[page 425]{Felix}, ${\bf v}_k<{\bf a}_s$, where ${\bf v}_k$ is the above mentioned vertex. Since $(\varphi\mid f_s)=(0,c)$ with $c=\min\{l_0^{\varphi}+1,l_0^{s}+1\}$, ${\bf v}_k$ is either ${\bf v}_j$ or ${\bf a}_s$. So we have that ${\bf v}_k={\bf v}_j$ and, as a consequence, ${\bf a}_s\in {\bf v}_j^{\geq}$.

When $ s \in J_3(\varphi)$,  it is clear that ${\bf v}_{j}\leq {\bf a}_s$ because $\varphi$ and $f_s$ share
its minimum terminal satellite
point (that is $P_j=P_{t_i}$).

Finally, assume  that $s \in J_4(\varphi)$ and suppose that ${\bf a}_s\not\in {\bf v}_j^{<}$. This means that ${\bf v}_j\leq {\bf a}_s$ and, in fact, ${\bf v}_j<{\bf a}_s$ (because $s\not\in J_3(\varphi)$). By \cite[page 425]{Felix} we have that $H(\varphi,\varphi')<H(\varphi,f_s)$, $\varphi'$ being also a curvette at $P_j$ different from $\varphi$. This implies, using Lemma \ref{torero}, that $c>\betabarra_1^{\varphi}/\betabarra_0^{\varphi}$, $c$ being the value such that $(\varphi\mid f_s)=(0,c)$. This is a contradiction and, thus, ${\bf a}_s\in {\bf v}_j^{<}$.

Therefore, the previous paragraphs and the fact that $\{J_1(\varphi), J_2(\varphi), J_3(\varphi), J_4(\varphi)\}$ is a partition of $\mathbb{J}_r$ conclude the proof of (a).

With respect to (b), our reasoning is analogous because the inclusion $J_1(\varphi)\subseteq {\bf v}_j^{<}$ (respectively, $J_{4,1}(\varphi)\subseteq {\bf v}_j^{<}$, $J_{4,2}(\varphi)\subseteq {\bf v}_j^{\geq}$) can be proved in a  similar way to the proof of $J_1(\varphi)\subseteq {\bf v}_j^{<}$ (respectively, $J_{4}(\varphi)\subseteq {\bf v}_j^{<}$, $J_{2}(\varphi)\cup J_{3}(\varphi)\subseteq {\bf v}_j^{\geq}$) of (a).

\end{proof}

Given a curvette $\varphi$ at a point of $\mathcal F$, the sets $J_3(\varphi)$ and $J_4(\varphi)$ (or $J_{4,1}(\varphi)$ and $J_{4,2}(\varphi)$ if $\varphi$ is nonsingular) are easy to compute only from the proximity graph of $C$. The following result, that is a straightforward consequence of the two previous lemmas, will allow us to detect, only by inspection of the proximity and dual graphs of $C$, the elements of the sets $J_k(\varphi)$, $k\in \{1,2\}$.

\begin{lem}
Let $\varphi$ be a curvette at a point $P_j\in {\mathcal F}$ and let $s\in \{1,\ldots,r\}$.
\begin{itemize}
\item[(a)] $s\in J_1(\varphi)$ if and only if $(\varphi \mid
f_s)=(0,c) \mbox{ with } c=\min \{l^{\varphi}_0+1,l^s_0+1\}\;$ and $\;{\bf a}_s\in \bv_j^{<}$.
\item[(b)] $s\in J_2(\varphi)$ if and only if $(\varphi \mid
f_s)=(0,c) \mbox{ with } c=\min \{l^{\varphi}_0+1,l^s_0+1\}\;$ and $\;{\bf a}_s\in \bv_j^{\geq}$.

\end{itemize}

\end{lem}

\begin{exa} {\rm Denote by $\varphi_j$ a curvette at the point $P_j$ in the constellation $\mathcal C$ of Example \ref{example1}, $1\leq j\leq 17$.  The partitions of $\mathbb{J}_8$ defined the curvettes $\varphi_j$ at the points in ${\mathcal T}\cup {\mathcal S}$ are given by the following sets:
\begin{description}
  \item[${\varphi}_2$] $J_1({\varphi}_2)=\emptyset, J_2({\varphi}_2)=\emptyset, J_3({\varphi}_2)=\emptyset, J_{4,1}({\varphi}_2)=\emptyset, J_{4,2}({\varphi}_2)=\{1,2,3,4,5,6,7,8\}$.
  \item[${\varphi}_4$] $J_1({\varphi}_4)=\{1,2\}, J_2({\varphi}_4)=\emptyset,   J_3({\varphi}_4)=\emptyset, J_{4,1}({\varphi}_4)=\{5,8\}, J_{4,2}({\varphi}_4)=\{3,4,6,7\}$.
  \item[$\varphi_7$] $J_1(\varphi_7)=\emptyset,\;\;\; J_2(\varphi_7)=\{2,3,4,6,7\},\;\;\; J_3(\varphi_7)=\{1\},\;\;\; J_{4}(\varphi_7)=\{5,8\}$.
  \item[$\varphi_8$] $J_1(\varphi_8)=\{1\},\;\;\; J_2(\varphi_8)=\{3,4,6,7\},\;\;\; J_3(\varphi_8)=\{2\},\;\;\; J_{4}(\varphi_8)=\{5,8\}$.
  \item[$\varphi_{11}$] $J_1(\varphi_{11})=\{1,2\},\;\;\; J_2(\varphi_{11})=\{4\},\;\;\; J_3(\varphi_{11})=\{3\},\;\;\; J_{4}(\varphi_{11})=\{5,8,6,7\}.$
  \item[$\varphi_{13}$] $J_1(\varphi_{13})=\{1,2,3\},\;\;\; J_2(\varphi_{13})=\emptyset,\;\;\; J_3(\varphi_{13})=\{4\},\;\;\; J_{4}(\varphi_{13})=\{5,8,6,7\}$.
  \item[$\varphi_{15}$] $J_1(\varphi_{15})=\{1,2\}, \;\;\; J_2(\varphi_{15})=\emptyset, \;\;\; J_3(\varphi_{15})=\emptyset, \;\;\;J_{4,1}(\varphi_{15})=\{3,4,5,8\},$\\ $J_{4,2}(\varphi_{15})=\{6,7\}$.
      \item[$\varphi_{17}$] $J_1(\varphi_{17})=\emptyset,\;\;\; J_2(\varphi_{17})=\{8\},\;\;\; J_3(\varphi_{17})=\{5\},\;\;\; J_{4}(\varphi_{17})=\{1,2,3,4,6,7\}$.
\end{description}

}
\end{exa}


Next two lemmas will allow us to prove Proposition \ref{kuwatilla}  and to show expressions of certain candidates for log-canonical threshold.

\begin{lem}\label{bombero}
If $\varphi$ is a curvette at $P_j\in {\mathcal F}$, then
${\overline \alpha}_{j}=\frac{\betabarra_0^{\varphi}+
\betabarra_1^{\varphi}}{\sum_{i=1}^r I(\varphi, f_i)}$.
\end{lem}

\begin{proof}
Consider the proximity matrix $\bold{P}$ of the configuration
${\mathcal C}$ associated with the minimal log resolution of the curve $C$. As we have said, the vector
$(a_1,a_2,\ldots,a_m)$ used in (\ref{WW}) can be computed as follows:
$(a_1,\ldots,a_m)^t=\bold{P}^{-1}(1,\ldots,1)^t$. The reason comes
from the fact that $K_{X/X_0} = \sum_{j=1}^m E_j^*$. Moreover, the entries of the $j$th row of $\bold{P}^{-1}$, $1 \leq j\leq
m$, are the multiplicities at the points of $\mathcal C$ of (the strict transforms of) a
curvette at $P_j$. Therefore, it holds that $a_j$ is the sum of the
multiplicities at the points of $\mathcal C$ of the strict
transforms of a curvette at $P_j$. If the curve that $\varphi$
provides, $C_\varphi$, is not regular, then the equality
$a_j+1=\betabarra_0^{\varphi}+\betabarra_1^{\varphi}$ follows from
\cite[Lemma 3.3.6]{Campillo} and the fact that
$\betabarra_0^{\varphi}$ and $\betabarra_1^{\varphi}$ coincide with
the  first two  characteristic exponents of $C_\varphi$ \cite[Proposition 4.3.5]{Campillo}. Otherwise
$\betabarra_0^{\varphi}=1$ and $\betabarra_1^{\varphi}$ is, as we
have defined above, the number of points through which the strict
transforms of $C_\varphi$ pass. Then, it is clear that the equality
$a_j+1=\betabarra_0^{\varphi}+\betabarra_1^{\varphi}$ is also true.

We have just proved that the numerator of ${\overline
\alpha}_{\varphi}$ is as stated. With respect to the denominator, it
holds that
$$(b_1,b_2,\ldots,b_m)^t:=\mathbf{P}^{-1}({\bf n}_1, {\bf n}_2, \ldots,{\bf n}_m)^t=
\sum_{i=1}^r \mathbf{P}^{-1}({\bf n}_{i1},{\bf n}_{i2},\ldots,{\bf
n}_{im})^t,$$
 where ${\bf n}_{ij}$ denotes the multiplicity at $P_j$ of the strict
transform of $C_i$. By Noether's formula, the $j$th component of the
vector $\mathbf{P}^{-1}({\bf n}_{i1},\ldots,{\bf n}_{im})^t$ is the
intersection multiplicity between $f_i$ and a curvette ${\varphi}$
at $P_j$. Therefore $b_j=\sum_{i=1}^r I(\varphi, f_i)$, which
concludes the proof.
\end{proof}

\begin{lem}\label{banderillero}
The candidates for log-canonical threshold given by the minimum terminal satellite points $P_{t_i}$, $1 \leq i \leq n$, can be computed as
$${\overline
\alpha}_{t_i}=\frac{\betabarra_0^i+\betabarra_1^i}{\sum_{s=1}^r
\delta_{is}},$$
where the values $\delta_{is}$ are defined as in the statement of Theorem \ref{gordo2}.
\end{lem}

\begin{proof}

Let $\varphi$ be a curvette at $P_{t_i}$ and set
$e_1^i:=\gcd(\betabarra_0^i, \betabarra_1^i)$. Using Noether's
formula and Lemma \ref{torero}, it is easy to deduce that
$$\betabarra_0^{\varphi}=\frac{\betabarra_0^i}{e_1^i}\;\;\;\;\;\mbox{and}
\;\;\;\; I(\varphi,f_s)=\frac{\delta_{is}}{e_1^i}.$$ Then, the
result follows from Lemma \ref{bombero}.

\end{proof}

Now we are ready to give a {\bf proof of Proposition
\ref{kuwatilla}}. Take a curvette $\varphi$ at a point $P_k\in
{\mathcal C}\setminus {\mathcal F}$ and let $i_0\in
\{1,2,\ldots,n\}$ such that $P_k\in {\mathcal C}_{i_0}$. It is
enough to show that ${\overline \alpha}_{k}\geq {\overline
\alpha}_{t_{i_0}}$. Let $\psi$ be a curvette at $P_{t_{i_0}}$. By the proof of
Lemma \ref{bombero}, there exists a positive integer $\epsilon$ such
that
$${\overline \alpha}_{k}=\frac{\betabarra_1^{\varphi}+
\betabarra_0^{\varphi}+\epsilon}{\sum_{s=1}^r I(\varphi, f_s)}$$
and
$${\overline \alpha}_{t_{i_0}}=\frac{\betabarra_1^{\psi}
+\betabarra_0^{\psi}}{\sum_{s=1}^r I(\psi, f_s)}.$$ Some
straightforward computations show that the inequality ${\overline
\alpha}_{k}\geq {\overline \alpha}_{t_{i_0}}$ holds if the
inequality
\begin{equation}\label{fua}(\betabarra_1^{\varphi}+\betabarra_0^{\varphi}
+\epsilon)I(\psi,f_s)\geq
(\betabarra_1^{\psi}+\betabarra_0^{\psi})I(\varphi,f_s)
\end{equation}
is true for all $s \in \mathbb{J}_r$.

Let us consider the partition of $\mathbb{J}_r$,
$\{J_l(\psi)\}_{l=1}^4$, given at the top of this section. Assume
first $s\in J_3(\psi)$ and thus $P_{t_{i_0}}=P_{t_s}$. Since $\pi$
is a log resolution of the curve $C_s$, it happens
$\pi^*C_s=\tilde{C}_s+\sum_{j=1}^m b_{js}E_j$ for some nonnegative
integers $b_{js}$ and the quotient
$\overline{\alpha}_{k}':=\frac{\betabarra_1^{\varphi}+\betabarra_0^{\varphi}+\epsilon}{I(\varphi,
f_s)}$ coincides with the candidate $\frac{a_k+1}{b_{k}}$ for log-canonical threshold of
$C_s$. But, by the proof of Lemma
\ref{banderillero}, the quotient
$\overline{\alpha}_{t_{i_0}}':=(\betabarra_1^{\psi}+\betabarra_0^{\psi})/{I(\psi,
f_s)}$ coincides with
$$\frac{\betabarra_0^s+\betabarra_1^s}{\betabarra_0^s\betabarra_1^s}
=\frac{1}{\betabarra_0^s}+\frac{1}{\betabarra_1^s},$$ which is the
log-canonical threshold of $C_s$ (see \cite{jar}). Therefore
$\overline{\alpha}_{k}'\geq \overline{\alpha}_{t_{i_0}}'$ and, then, inequality
(\ref{fua}) holds.

Now suppose that $s$ is in other set of the partition. By Lemma
\ref{torero} it happens that
\begin{equation*}
I(\varphi,f_s):= \left \{ \begin{array}{ll}
\betabarra_1^s\betabarra_0^\varphi &  \mathrm{if} \; \; s\in J_1(\psi) \\
 \betabarra_0^s\betabarra_1^\varphi  & \mathrm{if} \; \; s\in J_2(\psi)\\
  c\betabarra_0^s \betabarra_0^\varphi & \mathrm{if} \; \; s\in J_4(\psi),
\end{array} \right.
\end{equation*}
where $c$ is such that $(\varphi\mid f_s)=(0,c)$. Writing $\psi$
instead of $\varphi$, we obtain an analogous formula which can be
checked using the same argument.

Proposition \ref{kuwatilla} will be proved if we show that
\begin{equation}\label{fua2}(\betabarra_1^{\varphi}+\betabarra_0^{\varphi})
I(\psi,f_s)\geq
(\betabarra_1^{\psi}+\betabarra_0^{\psi})I(\varphi,f_s)\end{equation}
because this inequality implies that of (\ref{fua}). And
(\ref{fua2}) is equivalent to
$$\frac{\betabarra_1^{\varphi}}{\betabarra_0^{\varphi}}
\geq\frac{\betabarra_1^{\psi}}{\betabarra_0^{\psi}}\; \;
\left(\mbox{respectively, }
\frac{\betabarra_1^{\varphi}}{\betabarra_0^{\varphi}}\leq
\frac{\betabarra_1^{\psi}}{\betabarra_0^{\psi}},\;\;
\frac{\betabarra_1^{\varphi}}{\betabarra_0^{\varphi}}\geq
\frac{\betabarra_1^{\psi}}{\betabarra_0^{\psi}}\right),$$ whenever
$s \in J_1(\psi)$ (respectively, $s \in J_2(\psi)$,
$s \in J_4(\psi)$), which concludes the proof since
$$
\frac{\betabarra_1^{\varphi}}{\betabarra_0^{\varphi}} =
\frac{\betabarra_1^{\psi}}{\betabarra_0^{\psi}}.
$$
It only remains to add that this equality is true because $\varphi$ and $\psi$ share its minimum terminal satellite point $P_{t_{i_0}}$.

\subsection{Theorem \ref{gordo2}: auxiliary results and proof}
\label{32}

In this section we keep the notations introduced in the previous
ones.

Firstly and  by means of the following three lemmas, we will prove (1) of Theorem \ref{gordo2}. The first lemma shows that $\sigma$ is an increasing
function on the set ${\mathcal
V}_{\mathcal F}$ (under the  ordering $\leq$ defined before the definitions in (\ref{order+})) and it follows from the fact that ${\bf
v}_{k_1}^{<}\subseteq {\bf v}_{k_2}^{<}$ and ${\bf
v}_{k_2}^{\geq}\subseteq {\bf v}_{k_1}^{\geq}$, whenever ${\bf
v}_{k_1}, {\bf v}_{k_2} \in {\mathcal V}_{\mathcal F}$ and
${\bf v}_{k_1}< {\bf v}_{k_2}$.

\begin{lem}\label{minero}
Let ${\bf v}_{k_1}$ and ${\bf v}_{k_2}$ be two vertices in ${\mathcal V}_{\mathcal F}$ such that ${\bf v}_{k_1}< {\bf v}_{k_2}$. Then $\sigma({\bf v}_{k_1})\leq \sigma({\bf v}_{k_2})$.
\end{lem}

\begin{lem}

Let ${\bf v}_k\in {\mathcal V}$ and let $D: =\left\{{\bf v}_j\in {\mathcal
V}_{\mathcal F}\setminus \{{\bf v}_k\}\mid {\bf v}_j\geq {\bf v}_k  \mbox { {\rm and} } \sigma({\bf
v}_j)<0 \right\}$. Then, $D$ has, at most, one maximal element (with
respect to the order relation $\leq$).
\end{lem}

\begin{proof}
Reasoning by contradiction, assume the existence of two maximal
elements ${\bf v}_{j_1},{\bf v}_{j_2}$ in $D$.
One has that
$$\sigma({\bf v}_{j_a})=\sum_{{\bf a}_i\in {\bf v}_{j_a}^{<}} c_{j_{a}i}
\betabarra_0^i-\sum_{{\bf a}_i\in {\bf v}_{j_a}^{\geq}}\betabarra_0^i<0,$$
for $a=1,2$.

On the one hand we have that ${\bf v}_{j_1}^{\geq}\subseteq {\bf
v}_{j_2}^{<}$ and, therefore,
$$\sum_{{\bf a}_i\in {\bf v}_{j_2}^{\geq}}
\betabarra_0^i> \sum_{{\bf a}_i\in {\bf v}_{j_2}^{<}} c_{j_{2}i}
\betabarra_0^i\geq \sum_{{\bf a}_i\in {\bf v}_{j_1}^{\geq}}
c_{j_{2}i} \betabarra_0^i \geq \sum_{{\bf a}_i\in {\bf
v}_{j_1}^{\geq}} \betabarra_0^i.$$

On the other hand, by a symmetric reasoning, it holds that
$$\sum_{{\bf a}_i\in {\bf v}_{j_1}^{\geq }} \betabarra_0^i>
\sum_{{\bf a}_i\in {\bf v}_{j_2}^{\geq}} \betabarra_0^i,$$ because
${\bf v}_{j_2}^{\geq}\subseteq  {\bf v}_{j_1}^{<}$, which concludes
the proof.

\end{proof}

\begin{lem}\label{alonso}
There exists a vertex ${\bf v}_k\in {\mathcal V}_{\mathcal F}$
satisfying the conditions:
\begin{itemize}
\item[(a)] $\sigma ({\bf v}_j)< 0$ for all ${\bf v}_j\in [{\bf v}_1, {\bf v}_{k}]$ and
\item[(b)] $\sigma ({\bf v}_j)\geq 0$ for all ${\bf v}_j\in {\mathcal V}_{\mathcal F}\setminus [{\bf v}_1, {\bf v}_{k}]$.
\end{itemize}
\end{lem}

\begin{proof}

The result follows from the two preceding lemmas and the fact that $\sigma({\bf v}_1)<0$ because ${\bf v}_1^{<}=\emptyset$.

\end{proof}


Now we are ready to prove Part (1) of Theorem \ref{gordo2}. To do it, we need to define the concept of consecutive vertices in ${\mathcal V} \cup {\mathcal V}_{\mathrm{end}}$.

\begin{de}
{\rm
We will say that two vertices $\bv_{k_1}, \bv_{k_2}\in {\mathcal V}\cup {\mathcal V}_{\mathrm{end}}$ are \emph{consecutive} if  $\bv_{k_1}<\bv_{k_2}$ and $]{\bf v}_{k_1},{\bf v}_{k_{2}}[\cap {\mathcal V}=\emptyset$.
}
\end{de}

Consider the vertex  ${\bf v}_k\in {\mathcal V}_{\mathcal F}$ given by Lemma \ref{alonso}. If $\bv_1\not\in {\mathcal V}$ then the adapted degree of $\bv_1$ is 2 and, therefore, there is no arrow labeling it; this implies that $\sigma(\bv_2)=\sigma({\bv}_1)<0$ and, then, $\bv_k\not=\bv_1$.

Also, we claim that $\bv_k$ cannot be a vertex in ${\mathcal V}_{\mathrm{end}}\setminus {\mathcal V}$. Indeed, reasoning by contradiction, assume that $\bv_k\in {\mathcal V}_{\mathrm{end}}\setminus {\mathcal V}$. Then $\bv_k$ is a vertex whose adapted degree is, at most, 2. This implies that $\bv_k$ has, at most, one arrow ${\bf a}_i$ as a label (recall that $\bv_k\neq\bv_1$); moreover,  $C_i$ must be a nonsingular curve since $P_k$ is a free point in ${\mathcal F}$ and, therefore, $\betabarra_0^i=1$. Hence $\sigma(\bv_k)\geq \sum_{{\bf a}_s\in \bv_k^{<}} c_{ks}\betabarra_0^s-1\geq 0$ and this gives the desired contradiction.

So, \emph{if $\bv_k$ belongs to ${\mathcal V}_{\mathrm{end}}$ then it must belong to $\mathcal V$.}

Finally, we are going to show that \emph{$\bv_k$ belongs to $\mathcal V$}, which proves (1) of Theorem \ref{gordo2} because $\bv_k$ satisfies the conditions (a) and (b) given in the theorem's statement. Indeed, reasoning again by contradiction assume that $\bv_k\not\in {\mathcal V}$. Then there exist two consecutive vertices, $\bv_{j_1},\bv_{j_2}$ in ${\mathcal V}\cup {\mathcal V}_{\mathrm{end}}$, such that $\bv_{k}\in ]\bv_{j_1},\bv_{j_2}[$ (notice that, by the above proved claim, $\bv_k\not\in {\mathcal V}_{\mathrm{end}}$). Thus $\sigma(\bv_{k})=\sigma(\bv_{j_2})$ because $\sigma$ is constant along the path $]\bv_{j_1},\bv_{j_2}]$ (see Remark \ref{nota}) which contradicts Condition (b) of Lemma \ref{alonso}.\vspace{2mm}

Now, we are going to prove (2) of Theorem \ref{gordo2}. The following lemma will be of importance for our proof.

\begin{lem}\label{lemamadre}

Let $\bv_{k_1}, \bv_{k_2}$ be two consecutive vertices of ${\mathcal V}\cup {\mathcal V}_{\mathrm{end}}$. It holds that:
\begin{itemize}
\item[(a)] If $\sigma(\bv_{k_2})\leq 0$,  then $\overline{\alpha}_j\geq \overline{\alpha}_{k_2}$ for all $j$ such that $\bv_j\in [{\bf v}_{k_1},{\bf v}_{k_{2}}]$.
\item[(b)] If $\sigma(\bv_{k_2})\geq 0$,  then $\overline{\alpha}_j\geq \overline{\alpha}_{k_1}$ for all $j$ such that $\bv_j\in [{\bf v}_{k_1},{\bf v}_{k_{2}}]$.

\end{itemize}

\end{lem}

\begin{proof}

Let $\psi_1$ (respectively, $\psi_2$) be a curvette at $P_{k_1}$ (respectively, $P_{k_2}$).

If $l\in \{1,2\}$ and $\bv_{k_l}\in {\mathcal V}_{\mathcal S}\cup {\mathcal V}_{\mathrm{end}}$  then, by Lemma \ref{anthrax}, the non-empty sets in the family  $$\left\{J_1(\psi_l), J_{4,1}(\psi_l), J_{4,2}(\psi_l)\right\}$$ form a partition of $\mathbb{J}_r$. Taking into account Lemmas \ref{torero} and \ref{bombero}, it holds that
\begin{equation}\label{free}
\overline{\alpha}_{k_l}=\frac{d_{k_l}+1}{\sum_{s\in J_{4,1}(\psi_l)}c_{k_l s} \betabarra_0^s+\sum_{s\in J_1(\psi_l)} \betabarra_1^s+d_{k_l}\sum_{s\in J_{4,2}(\psi_l)} \betabarra_0^s}
\end{equation}
where, for each $P_e\in {\mathcal F}$, $d_e$ denotes the cardinality of the set $[\bv_1,\bv_l]\cap {\mathcal V}_{\mathrm{free}}$.

If, on the contrary, $\bv_{k_l}=\bv_{t_i}\in {\mathcal V}_{\mathcal T}$, then applying Lemmas \ref{torero} and \ref{bombero} again, we have that
\begin{equation}\label{satel}{\overline \alpha}_{k_l}={\overline \alpha}_{t_i}=\frac{\betabarra_0^{\psi_l}+\betabarra_1^{\psi_l}}{\sum_{s\in J_1(\psi_l)} \betabarra^{\psi_l}_0 \betabarra_1^s+ \sum_{s\in J_2(\psi_l)\cup
J_3(\psi_l)}
\betabarra_0^s\betabarra^{\psi_l}_1 + \sum_{s\in J_4(\psi_l)}
c_{k_{l} s}\betabarra_0^{\psi_l}\betabarra_0^s}.\end{equation}

We will distinguish four cases:\vspace{2mm}

\noindent \emph{Case 1}: $\bv_{k_1}, \bv_{k_2}\in {\mathcal V}_{\mathcal S}\cup {\mathcal V}_{\mathrm{end}}$. Let us consider a vertex $\bv_j\in [\bv_{k_1},\bv_{k_2}]$ and let $\varphi$ be a curvette at $P_j$. Clearly, $P_j$ is free and $\varphi$ is smooth. Moreover $I(\varphi,f_s)\leq I(\psi_2,f_s)$ for all $s\in \mathbb{J}_r$ and, by Lemma \ref{torero}, $I(\varphi,f_s)=d_j\betabarra_0^s$ whenever $s\in J_{4,2}(\psi_2)$. Therefore
$$\overline{\alpha}_j\geq \overline{\alpha}'_j:=\frac{d_{j}+1}{\sum_{s\in J_{4,1}(\psi_2)}c_{k_2 s} \betabarra_0^s+\sum_{s\in J_1(\psi_2)} \betabarra_1^s+d_{j}\sum_{s\in J_{4,2}(\psi_2)} \betabarra_0^s}.$$

Using (\ref{free}), straightforward computations show that the inequality $\overline{\alpha}'_j\geq \overline{\alpha}_{k_2}$ holds if and only if
$$\left(d_j-d_{k_2}\right)\left(\sum_{s\in J_{4,1}(\psi_2)} c_{k_2 s}\betabarra_0^s+ \sum_{s\in J_{1}(\psi_2)} \betabarra_1^s-\sum_{s\in J_{4,2}(\psi_2)} \betabarra_0^s\right)\geq 0$$
and, by Lemma \ref{ttt}, this inequality can be written in the following form:
$(d_j-d_{k_2})\sigma(\bv_{k_2})\geq 0$. Since $d_j\leq d_{k_2}$, we have that $\overline{\alpha}_j\geq \overline{\alpha}_{k_2}$ if $\sigma(\bv_{k_2})\leq 0$. So, we have proved Part (a) in Case 1.

To prove (b), let us consider a vertex $\bv_j$ and a curvette $\varphi$ as above. The facts $J_{4,2}(\psi_2)\subseteq J_{4,2}(\psi_1)$ and  $c_{k_1 s}=d_{k_1}$ for every $s\in \Delta:=J_{4,2}(\psi_1)\setminus J_{4,2}(\psi_2)$ allow us to write the equality (\ref{free}) as
\begin{equation}\label{free2}
\overline{\alpha}_{k_1}=\frac{d_{k_1}+1}{\sum_{s\in J_{4,1}(\psi_1)\cup \Delta}c_{k_1 s} \betabarra_0^s+\sum_{s\in J_1(\psi_1)} \betabarra_1^s+d_{k_1}\sum_{s\in J_{4,2}(\psi_2)} \betabarra_0^s}.
\end{equation}

On the other hand, from Lemma \ref{torero} one can deduce the following statements:

-- If $s\in J_{4,1}(\psi_1)$ then $I(\varphi,f_s)=c_{k_1 s}
\betabarra_0^s$, because $J_{4,1}(\psi_1)\subseteq J_{4,1}(\varphi)$ and $c_{j s}=c_{k_1 s}$.

-- If $s\in J_{1}(\psi_1)$ then $I(\varphi,f_s)=\betabarra_1^s$, because $J_1(\psi_1)\subseteq J_1(\varphi)$.

-- If $s\in J_{4,2}(\psi_2)$ then $I(\varphi,f_s)= d_j \betabarra_0^s$, because $J_{4,2}(\psi_2)\subseteq J_{4,2}(\varphi)$.

-- If $s\in \Delta$ then $I(\varphi,f_s)=d_{k_1}
\betabarra_0^s=c_{k_1 s} \betabarra_0^s$, because $\Delta\subseteq J_{4,1}(\varphi)$ and $c_{j s}=c_{k_1 s}=d_{k_1}$. \vspace{3mm}

Then, as a consequence of the previous statements and the fact that the non-empty sets in the family $\{ J_{1}(\psi_1),  J_{4,1}(\psi_1), J_{4,2}(\psi_2), \Delta \}$ are a partition of $\mathbb{J}_r$, it happens

$$\overline{\alpha}_j=\frac{d_{j}+1}{\sum_{s\in J_{4,1}(\psi_1)\cup \Delta}c_{k_1 s} \betabarra_0^s+\sum_{s\in J_1(\psi_1)} \betabarra_1^s+d_{j}\sum_{s\in J_{4,2}(\psi_2)} \betabarra_0^s}.$$
From this equality and that in (\ref{free2}), one can deduce that $\overline{\alpha}_j\geq \overline{\alpha}_{k_1}$ if and only if
$$\left(d_j-d_{k_1}\right)\left(\sum_{s\in J_{4,1}(\psi_1)\cup \Delta}c_{k_1 s} \betabarra_0^s+\sum_{s\in J_1(\psi_1)} \betabarra_1^s-\sum_{s\in J_{4,2}(\psi_2)} \betabarra_0^s\right)\geq 0.$$
Now, it is clear that $\bv_{k_2}^{<}={\bf v}_{k_1}^{<}\cup ({\bf v}_{k_1}^{\geq }\setminus {\bf v}_{k_2}^{\geq })$ and, using Lemma \ref{ttt}, one gets that
$$\bv_{k_2}^{<}=\{{\bf a}_s\mid s\in J_{4,1}(\psi_1)\cup \Delta\cup J_1(\psi_1)\}.$$
Moreover $c_{k_2 s}=c_{k_1 s}$ for all $s\in J_{4,1}(\psi_1)\cup \Delta$.
Thus the above inequality can be written as $(d_j-d_{k_1})\sigma(\bv_{k_2})\geq 0$ and hence Part (b) in Case 1 holds because $d_j\geq d_{k_1}$.\vspace{4mm}

\noindent \emph{Case 2: }$\bv_{k_1},\bv_{k_2}\in {\mathcal V}_{\mathcal T}$. Let us consider a vertex $\bv_j\in [\bv_ {k_1},\bv_{k_2}]$ and a curvette $\varphi$ at $P_j$. By Lemma \ref{torero}, it follows that $I(\varphi, f_s)\leq \min\{\betabarra_1^{\varphi}\betabarra_0^s, \betabarra_1^{s}\betabarra_0^{\varphi}\}$ for every $s\in \mathbb{J}_r$ and, moreover, $I(\varphi,f_s)=c_{k_{2} s} \betabarra_0^{\varphi}\betabarra_0^s$ for every $s\in J_4(\psi_2)$ (because the strict transforms of the curves  $C_s$ and that defined by $\varphi$ pass through the same free points). Therefore
$$\overline{\alpha}_j\geq \overline{\alpha}"_{j}:=\frac{\betabarra_0^{\varphi}+\betabarra_1^{\varphi}}{\sum_{s\in J_1(\psi_2)} \betabarra^{\varphi}_0 \betabarra_1^s+ \sum_{s\in J_2(\psi_2)\cup
J_3(\psi_2)}
\betabarra_0^s\betabarra^{\varphi}_1 + \sum_{s\in J_4(\psi_2)}
c_{k_{2} s}\betabarra_0^{\varphi}\betabarra_0^s}.$$
Taking into account (\ref{satel}), one can deduce that the inequality $\overline{\alpha}"_j\geq \overline{\alpha}_{k_2}$ holds if and only if
$$(\betabarra_1^{\varphi}\betabarra_0^{\psi_2}-
\betabarra_0^{\varphi}\betabarra_1^{\psi_2})\left(\sum_{s\in {J_1(\psi_2)}}
\betabarra_1^s- \sum_{s\in {J_2(\psi_2)}\cup J_3({\psi_2})} \betabarra_0^s+
\sum_{s\in J_4({\psi_2})} c_{k_2s}\betabarra_0^s \right)\geq 0.
$$
Finally, $\bv_j\leq \bv_{k_2}$ implies $\betabarra_1^{\varphi}/\betabarra_0^{\varphi}\leq \betabarra_1^{\psi_2}/\betabarra_0^{\psi_2}$ (see the proof of Lemma \ref{ttt}) and since, by Lemma \ref{ttt}, the second factor of the left-hand side of the above inequality is $\sigma(\bv_{k_2})$, we have proved (a) in Case 2 because $\alpha_j\geq \alpha_{k_2}$ whenever $\sigma(\bv_{k_2})\leq 0$.

To prove Part (b) in this case, one only need to make an analogous reasoning. In fact, one has to use the expression of $\overline{\alpha}_{k_1}$ in (\ref{satel}) instead of $\overline{\alpha}_{k_2}$  and the facts that $J_4(\psi_1)=J_4(\psi_2)$ and $c_{k_1 s}=c_{k_2 s}$ for all $s\in J_4(\psi_1)$. \vspace{4mm}

\noindent \emph{Case 3: } $\bv_{k_1}\in {\mathcal V}_{\mathcal S}\cup {\mathcal V}_{\mathrm{end}}$ and $\bv_{k_2}\in {\mathcal V}_{\mathcal T}$. Part (a) can be proved as in the previous case. To prove (b), keep the same notations and observe that  $J_1(\psi_2)=\emptyset$. Now, $\psi_1$ is nonsingular (by Lemma \ref{anthrax}) and
$${\overline \alpha}_{\psi_1}=\frac{d_{k_1}+1}{d_{k_1} \sum_{s\in J_2(\psi_2)
\cup J_3(\psi_2)} \betabarra_0^s +  \sum_{s\in J_4(\psi_2)} c_{k_1 s}
\betabarra_0^s },$$
by Lemma \ref{torero}. Moreover
$${\overline \alpha}_{j}\geq \overline{\alpha}^{(3)}_{j}:=
\frac{\betabarra_0^{\varphi}+\betabarra_1^{\varphi}}{\sum_{s\in
J_2(\psi_2)\cup J_3(\psi_2)}
 \betabarra_0^s\betabarra^{\varphi}_1 + \sum_{s\in J_4(\psi_2)}
  c_{k_1 s}\betabarra_0^{\varphi}\betabarra_0^s }.$$
As a consequence
$\overline{\alpha}_{j}^{(3)}\geq {\overline \alpha}_{\psi_1}$ if, and only if,
$$(\betabarra_1^{\varphi}-d_{k_1}\betabarra_0^{\varphi})\left(-\sum_{s\in
J_2(\psi_2)\cup J_3(\psi_2)} \betabarra_0^s+ \sum_{s\in J_4(\psi_2)}
c_{k_1 s}\betabarra_0^s\right)\geq 0.$$
And this inequality concludes the proof of (b) in this case after bearing in mind the following facts: $c_{k_1 s}=c_{k_2 s}$ for all $s\in J_4(\psi_2)$, the second factor of the left-hand side of the inequality is $\sigma(\bv_{k_2})$ by Lemma \ref{ttt} and $d_{k_1}\leq
\betabarra_1^{\varphi}/\betabarra_0^{\varphi}$ because the strict transforms of
$\varphi$ pass through, at least, $d_{k_1}$ free points of $\mathcal C$.
\vspace{4mm}

\noindent \emph{Case 4: } $\bv_{k_1}\in {\mathcal V}_{\mathcal T}$ and $\bv_{k_2}\in {\mathcal V}_{\mathcal S}\cup {\mathcal V}_{\mathrm{end}}$. Reasoning in a similar way as we made for Part (b) of Case 1, the equality
$$\overline{\alpha}_j=\frac{\betabarra_0^{\varphi}+\betabarra_1^{\varphi}}{\sum_{s\in J_{4,1}(\psi_2)}c_{k_2 s} \betabarra_0^{\varphi}\betabarra_0^s+\sum_{s\in J_1(\psi_2)} \betabarra_0^{\varphi}\betabarra_1^s+\sum_{s\in J_{4,2}(\psi_2)} \betabarra_1^{\varphi}\betabarra_0^s}
$$
is a consequence of the following three statements: \vspace{3mm}

-- If $s\in J_{4,1}(\psi_2)$ then $I(\varphi,f_s)=c_{k_2 s}\betabarra_0^{\varphi} \betabarra_0^s$, because $J_{4,1}(\psi_2)\subseteq J_{4}(\varphi)$ and $c_{j s}=c_{k_2 s}$.

-- If $s\in J_{1}(\psi_2)$ then $I(\varphi,f_s)=\betabarra_0^{\varphi} \betabarra_1^s$, because $J_1(\psi_2)\subseteq J_1(\varphi) \cup J_3(\varphi)$.

-- If $s\in J_{4,2}(\psi_2)$ then $I(\varphi,f_s)=\betabarra_1^{\varphi} \betabarra_0^s$. \vspace{3mm}

Thus using Equality (\ref{free}) and Lemma \ref{ttt}, we show that  $\overline{\alpha}_j\geq \overline{\alpha}_{k_2}$ if and only if $(\betabarra_1^{\varphi}-d_{k_2} \betabarra_0^{\varphi})\sigma(\bv_{k_2})\geq 0$, which proves (a) in this case because $d_{k_2} \geq \betabarra_1^\varphi/\betabarra_0^{\varphi}$.

Finally,  the  proof of (b) is analogous to that of Case 2 and our lemma is proved.

\end{proof}

\begin{exa}
{\rm
As a complement of the above proof, we check some of the results there used for the curve given in Example \ref{example1}. For a start, ${\mathcal V}={\mathcal V}_{\mathcal S}\cup {\mathcal V}_{\mathcal T}$, where ${\mathcal V}_{\mathcal S} = \{{\bf v}_2,{\bf v}_4, {\bf v}_{15}\}$ and ${\mathcal V}_{\mathcal T} = \{{\bf v}_7,{\bf v}_8, {\bf v}_{11}, {\bf v}_{13}, {\bf v}_{17}\}$.

Now consider the vertices ${\bf v}_{k_1} = {\bf v}_{4}$ and ${\bf v}_{k_2}= {\bf v}_{15}$, which are consecutive and both in ${\mathcal V}_{\mathcal S}$. Here $\psi_1 = \varphi_4$ and $\psi_2 = \varphi_{15}$.  $J_1(\psi_1)=\{1,2\}$, $J_{4,1}(\psi_1)=\{5,8\}$, $J_{4,2}(\psi_1)=\{3,4,6,7\}$ and $J_1(\psi_2)=\{1,2\}$, $J_{4,1}(\psi_2)=\{3,4,5,8\}$ and $J_{4,2}(\psi_2)=\{6,7\}$. As we have  said, both sets are a partition of $\mathbb{J}_8$. These vertices correspond to Part (b) in Case 1. Apart from $\psi_1$ and $\psi_2$, our lemma would also consider $\varphi = \varphi_{14}$ whose sets $J_i$ are $J_1(\varphi)=\{1,2\}$, $J_{4,1}(\varphi)=\{3,4,5,8\}$ and $J_{4,2}(\varphi)=\{6,7\}$. Moreover, $\Delta = J_{4,2}(\psi_1) \setminus J_{4,2}(\psi_2) = \{3,4\}$ and, as we have said, $\{J_{4,1}(\psi_1), J_{1}(\psi_1), J_{4,2}(\psi_2), \Delta\}$ is a partition of $\mathbb{J}_8$. Moreover, $ J_{4,1}(\psi_1) \subseteq J_{4,1}(\varphi)$, $ J_{1}(\psi_1) \subseteq J_{1}(\varphi)$, $ J_{4,2}(\psi_2) \subseteq J_{4,2}(\varphi)$ and $ \Delta \subseteq J_{4,1}(\varphi)$, as we stated.

To finish, we consider a situation corresponding to Case 4. It is ${\bf v}_{k_1} = {\bf v}_{8}$ and ${\bf v}_{k_2} = {\bf v}_{4}$. In this case, $\varphi = \varphi_8$ and $J_1(\varphi_8)=\{1\}, J_2(\varphi_8)=\{3,4,6,7\}, J_3(\varphi_8)=\{2\}, J_{4}(\varphi_8)=\{5,8\}$ and the statements given in Case 4 of the previous lemma for $\psi_2 = \varphi_4$ and $\varphi = \varphi_8$ hold.

}
\end{exa}

Next we are going to prove Statement (2) in Theorem \ref{gordo2}, which will finish the proof of this result and our paper. Let $P_j$ be a point in ${\mathcal F}$ and assume $\bv_k< \bv_j$. Then, Lemma \ref{lemamadre} proves that $\overline{\alpha}_j\geq \overline{\alpha}_k$ because the path $[\bv_k, \bv_j]$ is contained in a union of paths of the form $[\bv_{j_1}, \bv_{j_2}]$, $\bv_{j_1}$ and $\bv_{j_2}$ being two consecutive vertices in ${\mathcal V}\cup {\mathcal V}_{\mathrm{end}}$. Otherwise, a similar reasoning with the path $[\bv_j,\bv_k]$ also shows that $\overline{\alpha}_j\geq \overline{\alpha}_k$ and, by Proposition \ref{kuwatilla}, the log-canonical threshold of $C$ is $\overline{\alpha}_k$.

It only remains to obtain the expression of
${\overline \alpha}_k$ when $P_k\in {\mathcal S}$ because otherwise Lemma \ref{banderillero} provides that expression. Let  $C_{i_1}$ and $C_{i_2}$ be  as in the statement and denote by $\varphi$ a curvette at
$P_k$, then  the following chain of equalities finishes the proof:
$$\overline{\alpha}_k=\frac{c_{i_1i_2}+1}{\sum_{s=1}^r I(\varphi,f_s)}=
\frac{c_{i_1i_2}+1}{c_{i_1i_2} \betabarra_0^{i_1}+\sum_{i_1 \neq
s=1}^r c_{i_1i_2}
\betabarra_0^s}=\frac{\betabarra_0^{i_1}\betabarra_0^{i_2}+
I(f_{i_1},f_{i_2})}{\betabarra_0^{i_1}I(f_{i_1},f_{i_2})+\betabarra_0^{i_2}
\sum_{i_1 \neq s=1}^r I(f_{i_1},f_{s})}.$$ Notice that the second equality is a consequence of Lemma \ref{torero} and the last one
has been obtained by multiplying numerator and denominator by
$\betabarra_0^{i_1}\betabarra_0^{i_2}$.

\end{document}